\IfFileExists{./prepreamble-amspreprint.sty}{\RequirePackage[packages,theorems,changes]{./prepreamble-amspreprint}}{}

\makeatletter
\IfFileExists{./scoop-latex/scoop-style.sty}{\providecommand*{\input@path}{}\edef\input@path{{./scoop-latex/}\input@path}}{}
\makeatother
\PassOptionsToPackage{backend = biber}{biblatex}
\PassOptionsToPackage{commentmarkup = footnote}{changes}
\PassOptionsToPackage{inline}{enumitem}
\RequirePackage{algorithm}

\documentclass[english]{amsart}

\RequirePackage{biblatex}
\RequirePackage{ltxcmds}
\IfFileExists{./preamble-amspreprint.sty}{\RequirePackage[packages,theorems,changes]{./preamble-amspreprint}}{}

\usepackage{manuscript}

\IfFileExists{./postpreamble-amspreprint.sty}{\RequirePackage[packages,theorems,changes]{./postpreamble-amspreprint}}{}

\makeatletter
\@ifpackageloaded{changes}{
\definechangesauthor[name = {Ronny Bergmann}, color = {TolVibrantMagenta}]{RB}
\definechangesauthor[name = {Anton Schiela}, color = {TolVibrantBlue}]{AS}
\definechangesauthor[name = {Laura Weigl}, color = {TolVibrantOrange}]{LW}
}{}
\makeatother

\makeatletter
\@ifpackageloaded{hyperref}{%
	\hypersetup{
		pdftitle = {Newton's method into vector bundles Part II: : Application to Variational Problems},
		pdfauthor = {Laura Weigl, Ronny Bergmann, Anton Schiela},
		pdfkeywords = {Newton's method, Riemannian manifolds, vector bundle},
		pdfcreator = {Created using the Scoop Template Engine version 1.4.2.}
	}
}{
	\pdfinfo{
		/Title (Newton's method into vector bundles Part II: : Application to Variational Problems)
		/Author (Laura Weigl, Ronny Bergmann, Anton Schiela)
		/Subject ()
		/Keywords (Newton's method, Riemannian manifolds, vector bundle)
		/Creator (Created using the Scoop Template Engine version 1.4.2.)
	}
}
\makeatother

\title[Newton's method into vector bundles Part II]{Newton's method for nonlinear mappings into vector bundles Part II: Application to Variational Problems}

\author[L. Weigl]{Laura Weigl\orcidlink{0009-0006-4520-5968}}
\address[L. Weigl]{Department of Mathematics, University of Bayreuth, 95440 Bayreuth, Germany}
\email{\detokenize{laura.weigl@uni-bayreuth.de}}
\urladdr{https://num.math.uni-bayreuth.de/en/team/laura-weigl/}

\author[R. Bergmann]{Ronny Bergmann\orcidlink{0000-0001-8342-7218}}
\address[R. Bergmann]{Norwegian University of Science and Technology, Department of Mathematical Sciences, NO-7491 Trondheim, Norway}
\email{\detokenize{ronny.bergmann@ntnu.no}}
\urladdr{https://www.ntnu.edu/employees/ronny.bergmann}

\author[A. Schiela]{Anton Schiela\orcidlink{0000-0002-6959-2951}}
\address[A. Schiela]{Department of Mathematics, University of Bayreuth, 95440 Bayreuth, Germany}
\email{\detokenize{anton.schiela@uni-bayreuth.de}}
\urladdr{https://num.math.uni-bayreuth.de/en/team/anton-schiela/}

\date{\today}

\dedicatory{}

\begin{document}

\begin{abstract}
We consider the solution of variational equations on manifolds by Newton's method. These problems can be expressed as root finding problems for mappings from infinite dimensional manifolds into dual vector bundles. We derive the differential geometric tools needed for the realization of Newton's method, equipped with an affine covariant damping strategy. We apply Newton's method to a couple of variational problems and show numerical results.

\end{abstract}


\makeatletter
\ltx@ifpackageloaded{hyperref}{%
\subjclass[2020]{\href{https://mathscinet.ams.org/msc/msc2020.html?t=49Q99}{49Q99}}
}{%
}
\makeatother

\maketitle

\section{Introduction}%
\label{section:introduction}

Newton's method is one of the standard algorithms for solving nonlinear problems numerically. It can be formulated for mappings between very general spaces, classically between Banach spaces, and is thus applicable in particular to problems, involving nonlinear partial differential equations (PDEs). Many PDEs are formulated in weak form as variational problems. Here, we are looking for solutions of the equation $F(x)=0$, where $F \colon X \to E^*$ maps between a Banach space $X$ and a dual space $E^*$. With the help of test functions $e\in E$ this can be written equivalently as
\begin{equation}\label{eq:varLinear}
 \text{find } x\in X \text{ such that} \quad F(x)e = 0  \quad \text{ for all } e\in E.
\end{equation}
Some, but not all, variational problems arise from an energy principle of the form $\min_{x\in X} f(x)$, resulting in a stationarity condition of the form $f'(x)=0$ with $f'\colon X\to X^*$. 

The class of variational problems, however, is not restricted to mappings between linear spaces. Plenty of interesting problems involve nonlinear spaces, in particular, differentiable manifolds. One of the most well known is the problem of finding a geodesic that connects two given points by minimizing its Dirichlet energy among all connecting curves $\gamma\colon [0,T]\to \M$ on a Riemannian manifold $\M$:
 \begin{equation}\label{eq:Dirichlet}
  \min_{\gamma} \int_{[0,T]} \frac{1}{2} \lVert\dot \gamma(t) \rVert^2_{\gamma(t)}\, dt, \quad \gamma(0)=\gamma_0, \gamma(T)=\gamma_T.
\end{equation}
More generally, harmonic maps minimize a similar energy functional of mappings, defined on more general domains (cf.~\cite{helein2008harmonic}). Other important classes of examples are found in continuum mechanics~\cite{MarsdenRaitu:1994}, including cosserat materials \cite{sander2016numerical} and models of inextensible solids~\cite{rubin2000}, in micro-magnetism, or in density functional theory \cite{Altmann:2025}, to name just a few, see also \cite{bartels2015numerical,hardering2020geometric} for an account on the numerical analysis.
In this case, an energy principle of the form $\min_{x\in \X} f(x)$, where $\X$ is a differentiable, possibly infinite dimensional manifold, will lead to the equation $f'(x)=0$, where $f' \colon \X\to T^*\X$ maps into the cotangent bundle $T^*\X$ of $\X$.
Slightly more generally, we can formulate variational problems via mappings $F \colon\X\to \E^*$, where $\E$ is a vector bundle, and $\E^*$ its so-called dual bundle. Then our problem is of the form (which we will elaborate in detail, below):
\begin{equation}\label{eq:varVectorBundle}
\text{find } x\in \X \text{ such that } \quad  F(x)e = 0  \quad \text{ for all } e\in E_{y(x)},
\end{equation}
where $E_{y(x)} \subset \E$ is the appropriate linear space of test functions, which depends on $x$. In contrast to \eqref{eq:varLinear}, determining the correct space of test functions is part of the problem. In our example of a geodesic, $\X$ is a Sobolev manifold of mappings into $\M$, and the variational equation is written as follows:
\begin{equation}
    \label{eq:GeodesicsVariational}
  0 =\int_{[0,T]} \langle \dot \gamma(t),\dot {\delta \gamma}(t) \rangle_{\gamma(t)}\, dt \quad \text{ for all } \delta \gamma \in T_\gamma \X.
\end{equation}
In this work we consider the application of Newton's method to problems of the form~\eqref{eq:varVectorBundle}. We build upon the results of~\cite{WeiglSchiela:2024}, where Newton's method for mappings from a manifold $\X$ into a vector bundle $\E$ has been introduced and analyzed.
One of the main outcomes of this work is that a linear connection on $\E$ is required to render Newton directions well defined.
This generalizes the well known insight that covariant derivatives have to be employed in Newton methods for finding zeros of vector fields on Riemannian manifolds (cf. e.g. \cite{gabay1982minimizing, ArgyrosHilout:2009:1,FernandesFerreiraYuan:2017:1,louzeiro2025inexact}) or one-forms \cite{smith1993geometric,smith1994optimization}. For an account on Newton's method for vector fields and further literature we refer to \cite[Section 6]{AbsilMahonySepulchre:2008:1}.
In the semismooth context, Newton's method for 
vector fields on Riemannian manifolds has already been discussed in \cite{diepeveen2021inexact,DeOliveiraFerreira:2020:1,si2024riemannian}. Newton's method for shape optimization was considered in \cite{schulz2015towards}. 

To obtain convergence from remote initial guesses, Newton's method is typically equipped with a globalization strategy.
Especially in Newton's method applied in optimization on Riemannian manifolds some strategies including damping and trust region methods were introduced, see e.g. \cite{EfficientDampedNewton, absil2007trust}. In \cite[Sec. 6]{WeiglSchiela:2024} we presented an affine-covariant damping scheme. Here, a vector transport on $\E$ is needed in addition, which has to be consistent with the connection in a certain sense. In the current paper, we will discuss, how these concepts can be implemented, if the co-domain of our mapping is a dual vector bundle, so that variational equations on manifolds can be solved.

Our paper is structured as follows: we start by a brief recapitulation of the main differential geometric concepts, required to define Newton steps for mappings into general vector bundles in Section~\ref{sec:NewtonDualVectorBundles}.
We consider the specific case of Newton's method for mappings into dual vector bundles, and discuss the differential geometric tools, including dual connections, necessary to define Newton's method for variational equations.
This includes an affine covariant damping strategy that we present for dual vector bundles in Section~\ref{sec:AffineCovariantDamping}. As an important special case we consider embedded submanifolds in detail.
Finally, we turn to a couple of examples in Section~\ref{sec:applications}, two of which can be seen as a generalization of the geodesic problem, namely elastic geodesics in a force field and with obstacle avoidance. Finally, as an example from continuum mechanics we consider the simulation of an elastic inextensible rod.
%
%
\section{Preliminaries}%
\label{sec:Preliminaries}
In this section we recall some notations, definitions and basic properties of differential geometry and from~\cite{WeiglSchiela:2024} used throughout the work.
Let $\X$ be a \emph{Banach manifold}, i.\,e.\ a topological space $\X$ and a collection of charts $(U,\phi)$, where each chart $\phi\colon U \to \mathbb X$ maps an open subset $U$ of $\X$ homeomorphically into a real Banach space $\mathbb{X}$ with norm $\Vert\cdot\Vert_{\mathbb{X}}$.
Unless otherwise noted we will assume that the manifold is of class $C^1$, cf.~\cite[II, §1]{Lang:1999:1}.
We call mappings $F\colon \X \to \M$ between two Banach manifolds \emph{continuous, differentiable} or \emph{locally Lipschitz}, if their representations in charts, i.\,e.\ $\phi_{\M} \circ F \circ \phi_{\X}^{-1}$, have the respective property.
It can be shown that these properties are independent of the choice of charts.
The \emph{tangent spaces} of a manifold $\X$, denoted by $T_x\X$ for $x\in\X$, can be defined by using charts, see, e.g.,~\cite[II §2]{Lang:1999:1}. We will equip each $T_x \X$ with a norm that is equivalent to $\lVert\cdot\rVert_{\mathbb X}$ and thus renders $T_x \X$ a Banach space. In~\cite[Sec. 5.1]{WeiglSchiela:2024} we showed that, if the mapping $x\mapsto \lVert\cdot\rVert_x$ is sufficiently regular, a construction similar to the one, performed on Riemannian manifolds, leads to a metric distance on $\X$ that can be used to study convergence of Newton's method.

A map $F\colon\X\rightarrow\M$ between two $C^1$-manifolds $\X$ and $\M$ is called \emph{continuously differentiable}, if its representation in charts has the same property. Then this mapping induces a linear map $F^\prime(x)\colon T_{x}\X \rightarrow T_{F(x)}\M$, termed \textit{tangent map} or \textit{derivative} of $F$ at any $x\in\X$~\cite{Lang:1999:1}.
A map $W \colon X\rightarrow Y$ between normed linear spaces is called Newton-differentiable (or semi-smooth) (cf. \cite{Ulbrich:2011:1,Mifflin:1977:1,QiSun:1993:1}) at $x_*\in X$ with respect to a Newton-derivative $W' \colon X \to L(X,Y)$, if it satisfies the following property:
\begin{equation*}
 \lim_{x\to x_*} \frac{\lVert W'(x)(x-x_*)-(W(x)-W(x_*))\rVert_Y}{\lVert x-x_*\rVert_X} = 0.
\end{equation*}
A map $F\colon\X\rightarrow\M$ between two $C^1$-manifolds is called \emph{Newton-differentiable} (cf.~\cite[Def. 4.4]{WeiglSchiela:2024}), if its representation in charts is.
\subsection*{Vector bundles.}
Consider a \emph{vector bundle} $(\E, \Y, p)$, which will be assumed to be of class $C^1$, unless otherwise stated.
Its \textit{bundle projection} $p\colon\E \rightarrow \Y$ is a surjective $C^1$-map, which assigns to $e\in \E$ from the \emph{total space} $\E$ its base point $y=p(e)$ in the \textit{base manifold} $\Y$~\cite[III §1]{Lang:1999:1}.
Since the bundle projection contains all necessary information, we will use the mapping $p\colon\E\to\Y$ or simply $\mathcal E$, if the context is clear, to denote the vector bundle in the following. Vector bundles can be seen as manifolds with special structure: for each point $y\in \Y$ on the base manifold the so-called \textit{fibre} $E_y \coloneqq p^{-1}(y)$ over $y$ is a  topological vector space whose topology is induced by some norm, rendering $E_y$ complete. The zero element in $E_y$ is denoted by $0_y$.
A mapping $v\colon \Y \to \E$, which is a right inverse of the vector bundle projection, i.\,e.\ $p(v(y))=y$ for all $y\in \Y$, is called \textit{section} of $\E$. The set of all sections is denoted by $\Gamma(\E)$. The most prominent example for a vector bundle is the tangent bundle $\pi\colon T\X\to \X$ of a manifold $\X$ with fibres $T_x\X$. A section of $T\X$ is called a \textit{vector field}.

Vector bundles can be described via \emph{local trivializations}. For a Banach space $(\mathbb E,\lVert\cdot\rVert_{\mathbb E})$ and an open set $U \subset \Y$, these are diffeomorphisms $\tau\colon p^{-1}(U) \to U \times \mathbb E$, such that $\tau_y\coloneqq\tau|_{E_y} \in L(E_y,\mathbb E)$ are linear isomorphisms.
Thus, any element of $e\in \E$ can be represented as a pair $(y,e_\tau)\in \Y\times \mathbb E$ where $y=p(e)$ is the base point and the fibre part $e_\tau$ depends on the trivialization. For two trivializations $\tau $ and $\tilde \tau$ around $y\in\Y$ we can define smooth transition mappings $A\colon y \mapsto \tilde \tau_y\tau_y^{-1}$ with $A(y) \in L(\mathbb E,\mathbb E)$, and obtain $e_{\tilde \tau} = A(y)e_\tau$.

Consider the tangent space $T_e\E$ of the vector bundle $\E$ at $e$.
Similar to vector bundle elements, every element $\delta e \in T_e\E$ can be represented in local trivializations as a pair of elements $(\delta y,\delta e_\tau) \in T_y\Y \times \mathbb E$, where the tangential part $\delta y = p'(e)\delta e$ is given canonically.
The representation $\delta e_\tau$ of the fibre part, however, depends crucially on the chosen trivialization: if $e_{\tilde\tau} = A(y)e_\tau$, then $\delta e_{\tilde\tau} = A(y)\delta e_\tau+(A'(y)\delta y) e_\tau$ by the product rule.
In particular, there is no natural splitting of $\delta e$ into a tangential part and a fibre part.
For a vector bundle $p\colon\E\to\Y$ we can consider the corresponding \emph{dual vector bundle} $p^*\colon \E^* \to \Y$ which occurs naturally in many applications.
Its fibres are given by $E^*_y = (p^*)^{-1}(y) = L(E_y,\R)$, i.\,e.\ the dual space of $E_y$. A prominent example is the cotangent bundle $T^*\X$ of a manifold $\X$ which is the dual bundle of the tangent bundle $T\X$.

Given a mapping $s\colon\Y_1\to \Y_2$ between two manifolds and a vector bundle $p_2\colon\E_2\to \Y_2$, we denote the \emph{pullback bundle} $s^*\E_2 \coloneqq \{(x,e)\in \Y_1 \times \E_2\colon s(x) = p_2(e)\}$ of $\E_2$ via $s$ by $s^*p\colon s^*\E_2 \to \Y_1$. Its fibres are given by $(s^*E_2)_x \coloneqq \{x\} \times E_{2, s(x)}$.
\subsection*{Fibrewise linear mappings.}
Consider two vector bundles $p_1\colon \E_1 \to \Y_1$, $p_2\colon \E_2 \to \Y_2$. Let $y\in \Y_1$ and $z\in\Y_2$ and denote
\begin{equation*}
    L(E_{1,y},E_{2,z}) \coloneqq \{ l\colon E_{1,y}\to E_{2,z} \; \vert \; l \text{ continuous and linear}\}.
\end{equation*}
Given a mapping $s\colon\Y_1\to \Y_2$ we can collect all sets of linear mappings $L(E_{1,y},E_{2,s(y)})$ for all $y\in \Y_1$ in a set $\L(\E_1,s^*\E_2)$.
We observe that $p_\L\colon \L(\E_1,s^*\E_2) \to \Y_1$ is a vector bundle itself with fibres $L(E_{1,y},E_{2,s(y)})$.
In trivializations an element $H$ of $\L(\E_1,s^*\E_2)$ can be represented by a pair $(y, H_\tau)$ where $y\in\Y_1$ and $H_\tau \in L(\mathbb E_1,\mathbb E_2)$. If the co-domain is a fixed fibre, i.e, $E_z \cong \{z\}\times E_z$, we use the simplified notation $\L(\E_1,E_z)$ instead of the cumbersome $\L(\E_1,s_z^*(\{z\}\times E_z))$, where $s_z\colon\Y_1\to \{z\}$ is the constant map.
A section $S \in \Gamma(\L(\E_1,s^*\E_2))$ of the vector bundle $\L(\E_1,s^*\E_2)$, i.e, a mapping
\begin{equation*}
    S\colon \Y_1 \to \L(\E_1,s^*\E_2) \text{ with } p_\L \circ S=Id_{\Y_1}.
\end{equation*}
is called \emph{fibrewise linear mapping}, also known as \emph{vector bundle morphism}, see, e.g.,~\cite[III §1]{Lang:1999:1}.
$S$ is called \emph{locally bounded}, respectively \emph{differentiable}, if its representation in local trivializations
\begin{equation}\label{eq:FibrewiseTriv}
S_\tau\colon U \to L(\mathbb E_1,\mathbb E_2), \; S_\tau(y) = \tau_{\L(\E_1, s^*\E_2)}(S)\coloneqq \tau_{\E_2,s(y)}\circ S(y) \circ \tau_{\E_1,y}^{-1} \in L(\mathbb E_1,\mathbb E_2)
\end{equation}
has the respective property.

If $S$ is differentiable, then its derivative is a section $S'\in \Gamma(T\L(\E_1,s^*\E_2))$ of the tangent bundle $p_\L'\colon T\L(\E_1,s^*\E_2)\to T\Y_1$, i.e., a mapping
\begin{equation*}
S'\colon T\Y_1 \to T\L(\E_1,s^*\E_2), \text{ such that } p_\L' \circ S'=Id_{T\Y_1}.
\end{equation*}
From $S \in \Gamma(\L(\E_1,s^*\E_2))$ we can construct a mapping $\NE{S} \in C^1(\E_1,\E_2)$ as follows:
\begin{align*}
    \NE{S}\colon \E_1 &\to \E_2, \quad
                    e \mapsto \NE{S}(e)\coloneqq S(p_1(e))e.
\end{align*}
By this, we obtain a \textit{natural inclusion} of the sections of the bundle of fibrewise linear mappings into the $C^1$-mappings:
\begin{equation}
    \label{eq:defNaturalInclusion}
\varphi\colon \Gamma(\L(\E_1, s^*\E_2)) \to C^1(\E_1, \E_2), \quad S \mapsto \NE{S}.
\end{equation}
Comparison of the derivatives $S'\in \Gamma(T\L(\E_1,s^*\E_2))$ and $\NE{S}' \in \Gamma(\L(T\E_1,S^*T\E_2))$ in trivializations yields by the product rule
\begin{equation}
    \label{trivializationDerivative}
    \NE{S}_\tau'(y,e_\tau)(\delta y, \delta e_\tau) = S_\tau(y)\delta e_\tau+(S_\tau'(y)\delta y) e_\tau  \quad \text{ for all } (\delta y,\delta e_\tau) \in T_{y}\Y_1\times \mathbb E_1.
\end{equation}
\begin{example}
Given a twice differentiable mapping $F\colon \X \to \M$, the derivative $S\coloneqq F'\in \Gamma(\L(T\X,F^*T\M))$ is a fibrewise linear mapping with $\langle F'\rangle \in C^1(T\X,T\M)$.
\end{example}
\subsection*{Vector transports.} Special fibrewise linear mappings are the so-called \emph{vector transports}.
\begin{definition}
    For a vector bundle $p\colon\E\rightarrow\Y$ and $y \in \Y$ we define a \emph{vector transport} as a section $\VTforward{y} \in \Gamma(\L(\Y \times E_{y}, \E))$, i.e.
    \begin{equation*}
        \VTforward{y}\colon \Y \to \L(\Y \times E_{y}, \E), \ \text{with } p_\L(\VTforward{y}(\hat{y})) = \hat{y} \;\text{ for all } \hat{y}\in\Y
    \end{equation*}
    with the properties that $\VTforward{y}(y)=Id_{E_{y}}$ and $\VTforward{y}(\hat y)$ is invertible for all $\hat y\in \Y$. \\
    We define a \emph{vector back-transport} as a section $\VTback{y}\in \Gamma(\L(\E, E_{y}))$, i.e.
    \begin{equation*}
        \VTback{y}\colon \Y \to \L(\E, E_y), \ \text{with } p_\L(\VTback{y}(\hat{y})) = \hat{y} \; \text{ for all } \hat{y}\in\Y
    \end{equation*}
    with the property that $\VTback{y}(y)=Id_{E_{y}}$ and $\VTback{y}(\hat y)$ is invertible for all $\hat y\in \Y$.
    \end{definition}
In the literature vector transports are often defined as mappings $V(\cdot,\cdot)\colon \Y\times \Y \to \L(\E,\E)$ with $V(y,\hat y)\in L(E_y,E_{\hat y})$. These definitions admit vector transports as well as vector back-transports by $\VTforward{y}(\hat y)=V(y,\hat y)=\VTback{\hat{y}}(y)$. Nevertheless, the derivatives of these two objects show subtle, but important differences.
\newline
\subsection*{Connections on vector bundles.}
The concept of a \emph{connection} on a vector bundle $p\colon \E\to \Y$ is fundamental to differential geometry. It adds a geometric structure imposed on $\E$ that describes in an infinitesimal way, how neighboring fibres are related. Considering the tangent bundle $T\X$ of a Riemannian manifold $\X$ the most prominent example is the \textit{Levi-Civita connection} (see, e.g.,~\cite[VIII §4]{Lang:1999:1}).
Connections give rise to further concepts like the covariant derivative, curvature, geodesics, and parallel transports.
They are described in the literature in various equivalent ways (cf., e.g.,~\cite[IV §3 or X §4 or XIII]{Lang:1999:1}).
Here we choose a formulation that emphasizes the idea of a connection map $Q_e\colon T_e\E \to E_{p(e)}$, which then induces a splitting of $T_e\E$ into a vertical and a horizontal subspace.
Using the bundle projection $p \colon \E\to \Y$, it is well known that the kernel of the derivative $p^\prime \in \Gamma(\L(T\E, p^*T\Y))$ canonically defines the \emph{vertical subbundle} $V\E$ of $\E$~\cite[IV §3]{Lang:1999:1}.
Its fibres are closed linear subspaces, called the \textit{vertical subspaces} $\mathrm{Vert}_e \coloneqq \ker p^\prime(e) \subset T_e\E$.
We can identify $E_{p(e)} \cong  \mathrm{Vert}_e$ canonically by the isomorphism
$|_{E_{p(e)}}\colon E_{p(e)} \rightarrow \mathrm{Vert}_e$, given by $w \mapsto \frac{d}{dt}(e+tw)\vert_{t=0}$. Elements of $\mathrm{Vert}_e$ are represented in trivializations by pairs of the form $(0_{y},\delta e_\tau)\in T_y\Y \times \mathbb E$.
We define a \emph{connection map} as a section
$Q\in \Gamma(\L(T\E,p^*\E))$, i.e., for each $e\in\E$ we get a continuous linear map $Q_e \in L(T_e\E,E_{p(e)})$,
with the additional property that $Q_e |_{E_{p(e)}} = Id_{E_{p(e)}}$.
The kernels $\mathrm{Hor}_e \coloneqq \ker Q_e$ are called \emph{horizontal subspaces}, collected in the \emph{horizontal subbundle} $H\E$.
While $V\E$ is given canonically, a choice of $Q$ (or equivalently $H\E$), imposes additional geometric structure on $\E$. Using the identification $\mathrm{Vert}_e \cong E_{p(e)}$, we can view $Q$ as a fibrewise ``projection'' onto $\E$, with $\langle Q \rangle\colon T\E \rightarrow \E$.
In trivializations $(y, e_\tau) \in \Y\times \mathbb E$ for $e\in\E$ and $(\delta y, \delta e_\tau) \in T_y\Y \times \mathbb E$ for $\delta e\in T_e\E$, the representation for a connection map $Q$ can be written in the form
\begin{equation}\label{eq:Qtriv}
    Q_e\delta e \sim Q_{e,\tau}(\delta y,\delta e_\tau)=\delta e_\tau - B_{y,\tau}(e_\tau)\delta y,
\end{equation}
where $B_{y,\tau}\colon \mathbb E \to L(T_y\Y,\mathbb E)$ assigns a linear mapping $B_{y,\tau}(e_\tau)$ to each $e_\tau \in \mathbb E$.
Since $Q_{e,\tau}(0_y,\delta e_\tau)=\delta e_\tau$, we see that $Q_{e,\tau}$ indeed represents a projection onto $\mathrm{Vert}_e$.
To reflect the linearity of the fibres we want the connection to be \emph{linear} (fibrewise with respect to $e$). To this end, we consider the fibrewise scaling ${m_s\colon \E\rightarrow\E, e \mapsto se}$ and require the condition $Q \circ m_s^\prime = m_s \circ Q$ for all $s\in\mathbb{R}$,
which reads in trivializations:
\begin{equation*}
   Q_{se,\tau}(\delta y,s\delta e_\tau)=sQ_{e,\tau}(\delta y,\delta e_\tau) \quad \text{ for all } s\in \mathbb R.
\end{equation*}
It can be shown that $Q$ is linear, if and only if the mapping $e_\tau \mapsto B_{y,\tau}(e_\tau)$ is linear, or put differently, the mapping $(\delta y,e_\tau) \mapsto B_{y,\tau}(e_\tau)\delta y$ is bilinear.
If $\E=T\X$ in classical Riemannian geometry, where $Q$ is given by the Levi-Civita connection, the bilinear mapping is represented by Christoffel symbols, and it is symmetric.
As we have seen in ~\cite[Lemma 2.2]{WeiglSchiela:2024} connection maps $Q_e$ at $e\in\E$ can be derived from differentiable vector back-transports $\VTback{p(e)}$ by setting $Q_e = \NE{\VTback{p(e)}}'(e)$.
\begin{definition}
\label{def:consistentConnectionPrimal}
 Let $e\in \E$ and $y=p(e)$. We call a vector back-transport $\VTback{y} \in \Gamma(\L(\E, E_y))$  \emph{consistent} with a connection map $Q \in \Gamma(\L(T\E, p^*\E))$ at $e$, if $Q_e = \NE{\VTback{y}}^\prime(e)$.
\end{definition}
\subsection*{Newton directions for mappings into vector bundles.} Consider a Newton-differentiable mapping $F\colon\X\to\E$ between a Banach manifold $\X$ and a vector bundle $p\colon\E\to\Y$ with Newton-derivative $F'\colon T\X \to T\E$. By the composition
\begin{equation}
    \label{eq:YBasePoint}
    y \coloneqq p\circ F\colon \X \to \Y, \quad y(x) \coloneqq p(F(x))\in \Y,
\end{equation}
we can compute the base point $y(x) \in \Y$ of $F(x)\in \E$ for given $x\in \X$.
In~\cite{WeiglSchiela:2024}, we presented a Newton method to solve the following root finding problem:
\begin{equation*}
    F(x) = 0_{{y(x)}}.
\end{equation*}
In contrast to classical root finding problems, the linear space $E_{y(x)}$ in which $F(x)$ is evaluated now depends on the iterate $x\in\X$. Hence, the linear space $E_{y(x)}$ changes during the iteration. We have seen that connection maps $Q\in \Gamma(\L(T\E, p^*\E))$ allow us to derive suitable Newton directions by solving the following linear operator equation:
\begin{equation}
    \label{eq:NewtonEquationPrimal}
    Q_{F(x)} \circ F^\prime(x)\delta x + F(x) = 0_{{y(x)}}.
    \end{equation}
    If $Q_{F(x)} \circ F^\prime(x)$ is invertible, then the Newton direction $\delta x \in T_x\X$ is given as the unique solution of this equation, and $\delta x=0$ holds if and only if $x$ is a zero of $F$, i.\,e.\ $F(x)=0_{y(x)}$.
%
%
\section{Newton's Method for variational equations on manifolds}
\label{sec:NewtonDualVectorBundles}
In mathematical physics we frequently encounter variational problems, which often yield differential equations in weak form. If these problems are posed on manifolds, \emph{dual vector bundles} can be used for their treatment.
In this section we discuss Newton's method for mappings into general dual vector bundles.
To do this, we will adapt the method presented in~\cite{WeiglSchiela:2024} to this case.
\subsection{Dual vector bundles}
Let $p\colon\E \to \Y$ be a vector bundle. Consider its dual bundle $p^*\colon\E^* \to \Y$ where each fibre $E^*_y = (p^*)^{-1}(y) = L(E_y,\R)$ is the dual space of $E_y$, so $\E^*=\L(\E,\R)$.
If $\tau_\E\colon \E \to \mathbb E$ is a trivialization of $\E$ with $e_\tau=\tau_{\E,y}(e)$ for $e\in E_y$, then by~\eqref{eq:FibrewiseTriv} dispensing with a trivialization of the co-domain $\R$ of $\ell\in E_y^*=L(E_y,\R)$,  we obtain a local trivialization $\tau_{\E^*}\colon \E^*\to \mathbb E^*$ of the dual bundle $\E^*$ by the mapping $l_\tau=\tau_{\E^*}(l) \coloneqq l \circ \tau_{\E,y}^{-1}$, so that for $e\in E_y$ with representation $e_\tau$:
\begin{equation}
    \label{eq:leTriv}
    l_\tau(e_\tau)=l(e).
\end{equation}
We denote the fibrewise dual pairing for $\ell\in \Gamma(\E^*)$ and $e\in\Gamma(\E)$ by $\ell(e)\colon \Y \rightarrow \R, \; y \mapsto \ell(y)e(y) \in \R$. Its derivative at a point $y\in \Y$ is given by $\ell(e)'(y) \in  L(T_y\Y,\R)$.
\subsection{Variational problems on manifolds}
Consider a  mapping $F\colon\X \rightarrow \E^*$ between a Banach manifold $\X$ and a dual vector bundle $p^*\colon \E^* \to \Y$. According to~\eqref{eq:YBasePoint} we denote the base point of $F(x)$ by $y(x) = p^*(F(x))$ for every $x\in\X$. Then we would like to find the solution of the following root finding problem:
\begin{equation*}
    F(x) = 0_{y(x)}^*.
\end{equation*}
Since this is a problem in a dual space it can equivalently be stated as a \emph{variational problem}, where $e\in E_{y(x)}$ plays the role of a test function:
\begin{equation*}
    F(x)(e) = 0 \quad \text{ for all } e\in E_{y(x)}.
\end{equation*}
\begin{example}
    Consider a twice differentiable function $f\colon\X \rightarrow \R$, defined on a $C^2$-Banach manifold $\X$. In order to find a critical point of $f$, we consider the \emph{covector field} $f^\prime\colon \X \rightarrow T^*\X$ which is a mapping into the cotangent bundle $\pi^*\colon T^*\X \to\X$, i.\,e.\ a dual vector bundle. Then, we want to compute a point $x\in\X$ such that
\begin{equation*}
    f'(x) = 0_x^* \in T_x\X^*.
\end{equation*}
\end{example}
In contrast to classical variational problems, the space of test functions $E_{y(x)}$ now depends on the state $x\in\X$.

To derive Newton's method for this problem, we need to define a Newton equation similar to~\eqref{eq:NewtonEquationPrimal} to derive a Newton direction $\delta x\in T_x\X$ as well as an update formula for the \emph{Newton step} $x_+\in\X$. In particular, the definition of the Newton equation requires a \emph{dual connection map}.
\subsection{Dual connection maps}\label{sec:DualConnectionMaps}
Consider a connection map $Q$ on the vector bundle $p\colon\E\to\Y$. Then a corresponding dual connection map $Q^*$ on the dual vector bundle $p^*\colon\E^* \to \Y$ can be defined as follows~\cite[Chap. 4]{Lee:2018:1}:
\begin{definition}\label{def:dualconnection}
    Let $Q$ be a connection map on $p\colon\E \rightarrow \Y$.
    The \emph{dual connection map} $Q^* \in \Gamma(\L(T\E^*,(p^*)^*\E^*))$ on  $p^*\colon\E^* \rightarrow \Y$ corresponding to $Q$ is defined by
    \begin{equation}\label{eq:defdualconnection}
        (Q_{\ell(y)}^*\circ \ell^\prime(y)\delta y)(e(y)) \coloneqq \ell(e)'(y) \delta y - \ell(y)(Q_{e(y)}\circ e^\prime(y) \delta y)
    \end{equation}
    for all $e \in \Gamma(\E)$, $\ell \in \Gamma(\E^*)$ and for all $\delta y \in T_y\Y$ where $y = p(e) = p^*(\ell)$.
\end{definition}
To show that this is a valid definition, we have to find a representation of $Q^*$ in trivializations. Suppose that the connection map $Q$ is represented as in~\eqref{eq:Qtriv} by
\begin{equation*}
    Q_e\delta e \sim Q_{e,\tau}(\delta y, \delta e_\tau) = \delta e_\tau - B_{y,\tau}(e_\tau)\delta y \in \mathbb E.
\end{equation*}
Recall that in trivializations $l \in \E^*$ is represented by $(y, l_\tau) \in \Y \times \mathbb E^*$ and $\delta l \in T_l\E^*$ by $(\delta y, \delta l_\tau) \in T_y\Y \times \mathbb E^*$.
\begin{proposition}
    For $l \in \E^*$ and $y=p^*(l)$ the linear mapping $Q_l^* \in L(T_l\E^*, E_{y}^*)$ is given by
    \begin{equation}\label{eq:Qstartriveq}
 (Q^*_l\delta l)(v)=Q^*_{l,\tau}(\delta y,\delta l_\tau)(v_\tau)=\delta l_\tau(v_\tau) + l_\tau\circ B_{y,\tau}(v_\tau)\delta y
\end{equation}
for all $v\in E_y$, $\delta l \in T_l\E^*$ and $\delta y \in T_y\Y$.
In particular,~\eqref{eq:defdualconnection} yields a valid definition of a connection on $\E^*$.
\end{proposition}
\begin{proof}
Let $e\in \Gamma(\E)$ and $\ell \in \Gamma(\E^*)$. We will show that~\eqref{eq:defdualconnection} coincides with~\eqref{eq:Qstartriveq} if we set $l = \ell(y)$, $v = e(y)$ and $\delta l = \ell'(y)\delta y$.
Indeed, using the above notations and the representation of $e'(y)\delta y$ by $(\delta y,\delta e_\tau)$ we obtain by the product rule and cancellation of the term $\ell_\tau(\delta e_\tau)(y)\delta y$:
\begin{align*}
 ((Q_{\ell,\tau}^*\circ \ell_\tau^\prime)(y)\delta y)(e_\tau(y))&=\ell_\tau(e_\tau)'(y)  - \ell_\tau(y)(\delta e_\tau-B_{y,\tau}(e_\tau(y))\delta y)\\
 &=\delta \ell_\tau(e_\tau(y))+\ell_\tau(y)(B_{y,\tau}(e_\tau(y))\delta y).
\end{align*}
Thus, $(Q_{\ell(y)}^*\circ \ell^\prime(y)\delta y)(e)$ does not depend on $e_\tau'(y)$, but on $e_\tau(y)$ only, so indeed $Q_{\ell(y)}^*\circ \ell^\prime(y)\delta y \in E_y^*$.
\end{proof}
\begin{remark}
Classically, connections are closely related to covariant derivatives:
    \begin{equation*}
     \nabla_{\delta y}e(y) \coloneqq Q_{e(y)}\circ e'(y)\delta y \mbox{ for } y \in \Y, e \in \Gamma(\E), \delta y\in T_y\Y,
    \end{equation*}
  in which case Definition~\ref{def:dualconnection} yields a product rule:
  \begin{equation*}
    \ell(e)'(y)\delta y =\nabla_{\delta y}^* \ell(y)(e) + \ell(\nabla_{\delta y} e(y)).
   \end{equation*}
\end{remark}
\subsection{Newton equation}\label{subsec:NewtonEquation} With the help of a dual connection map we can now define a Newton equation similar to~\eqref{eq:NewtonEquationPrimal}. Let us consider the mapping $F\colon\X\to\E^*$ again. The tangent map of $F$ is a mapping $F^\prime \in \Gamma(\L(T\X,F^*T\E^*))$. Since $F^\prime(x)\delta x \in T_{F(x)}\E^*$ and $F(x)\in E^*_{y(x)}$, these two quantities cannot be added, and thus the classical Newton equation $F^\prime(x)\delta x + F(x) = 0$ is not well defined.
To resolve this issue, we now use a dual connection map $Q^* \in \Gamma(\L(T\E^*,p^*\E^*))$. With its help we can define a linear mapping at $x\in \X$, which maps into the suitable space $E^*_{y(x)}$, as follows:
\begin{equation*}
    Q^*_{F(x)} \circ F'(x)\colon T_x \X \to E^*_{y(x)}.
\end{equation*}
Then the \emph{Newton equation} is well defined as the following linear operator equation
\begin{equation}
    \label{eq:NewtonEquationDual}
Q_{F(x)}^* \circ F^\prime(x)\delta x + F(x) = 0^*_{{y(x)}}
\end{equation}
which is equivalent to
\begin{equation*}
    (Q_{F(x)}^*\circ F'(x)\delta x+F(x))(e)=0 \quad \text{ for all } e\in E_{y(x)}.
\end{equation*}
If $Q^*_{F(x)} \circ F^\prime(x)$ is invertible, then the Newton direction $\delta x \in T_x\X$ is given as the unique solution of this equation.
Setting $\ell=F(x)$, we obtain by application of~\eqref{eq:Qstartriveq} with representations $F'(x)\delta x \sim (y'(x)\delta x, (F'(x)\delta x)_\tau)$:
\begin{equation}
    \label{eq:NewtonMapTriv}
 (Q_{F(x)}^*\circ F'(x)\delta x)(e)=(F'(x)\delta x)_\tau(e_\tau)+F_\tau(x)\circ B_{
 y(x),\tau}(e_\tau)y'(x)\delta x
\end{equation}
for all $e\in\E$ with $e_\tau=\tau_{\E,y}(e)\in \mathbb E$.

\subsection{Definition of Newton's method}\label{sec:DefNewton}
As we have seen in Sec.~\ref{subsec:NewtonEquation} dual connection maps can be used to obtain a well defined Newton equation~\eqref{eq:NewtonEquationDual}:
\begin{equation*}
    Q^*_{F(x)} \circ F^\prime(x)\delta x + F(x) = 0_{{y(x)}}^*.
\end{equation*}
The solution of this equation yields the Newton direction $\delta x\in T_x\X$.
Since the iterate $x\in \X$ and the Newton direction $\delta x\in T_x\X$ cannot be added, the classical additive Newton update $x_+=x+\delta x$ is also not well defined. To obtain a new iterate $x_+$, we have to map $\delta x\in T_x\X$ back to the manifold $\X$ in an appropriate way. This can be done by applying a \textit{retraction} (cf., e.g.,~\cite[Chap.~4]{AbsilMahonySepulchre:2008:1} or~\cite[Sec.~3.6]{Boumal:2023:1}) to the Newton direction.
\begin{definition}[Retraction]
    A $C^1$-\emph{retraction} on a manifold $\X$ is a $C^1$-mapping $R:T\X \rightarrow \X$, where its restriction $R_x\colon T_x\X\to \X$ to a fixed $x\in \X$ satisfies the following properties:
    \begin{itemize}
    \item[(i)] $R_x(0_x)=x$
    \item[(ii)] $R_x^\prime(0_x)=Id_{T_x\X}$
    \end{itemize}
\end{definition}
Thus, after successfully computing the Newton direction $\delta x\in T_x\X$ we use a local retraction to generate the new iterate, i.\,e.\ we perform a \textit{Newton step}, by
\begin{equation*}
    x_+ \coloneqq R_x(\delta x).
\end{equation*}
We summarize our results in Algorithm~\ref{alg:Newton}. It can be shown that under suitable semismoothness assumptions the method converges locally with a superlinear rate. For a detailed proof see~\cite[Prop. 5.10]{WeiglSchiela:2024}.
\begin{algorithm}[h]
    \caption{Newton's method on vector bundles}
    \label{alg:Newton}
        \begin{algorithmic}
            \REQUIRE $x_0$ (initial guess)
            \FOR {$k = 0,1,2,\ldots$}
                \STATE $\delta x_k \leftarrow Q_{F(x_k)}^* \circ F^\prime (x_k)\delta x_k + F(x_k) = 0_{{y(x_k)}}^*$
                \STATE $x_{k+1} = R_{x_k}(\delta x_k)$
            \ENDFOR
    \end{algorithmic}
\end{algorithm}

\begin{remark}
  In the case of a minimization problem $\min_{x\in \X} f(x)$ we obtain the variational equation $f'(x)=0_x^*$, and the corresponding Newton equation is formulated in the cotangent space $T_x\X^*$:
  \begin{equation}\label{eq:NewtonMinDer}
  Q^*_{f'(x)}\circ f''(x)\delta x+f'(x)=0^*_x.
  \end{equation}
  This is in accordance with the approach, discussed in \cite{smith1994optimization}.
  
  On Riemannian manifolds, where a Riesz-isomorphism $\mathcal R\in \Gamma(\L(T\X,T^*\X))$ is available, one frequently considers instead the equivalent problem $\mathrm{grad}\, f(x)=0_x$, where $\mathrm{grad}f(x)=\mathcal R_x^{-1}(f'(x)) \in T_x\X$ is the Riemannian gradient of $f$. This yields a Newton equation on the tangent space $T_x\X$ (cf. \cite[Section 6]{AbsilMahonySepulchre:2008:1}), using the covariant derivative, induced by the Levi-Civita connection $Q$ on $\X$:
  \begin{equation}\label{eq:NewtonMinGrad}
  \nabla\, \mathrm{grad}f(x)\delta x+\mathrm{grad}f(x)=0_x.
  \end{equation}
  The linear operator $\nabla\, \mathrm{grad}f(x):=Q_{\mathrm{grad}f(x)}\circ (\mathrm{grad}f)'(x)\in L(T_x\X,T_x\X)$ is called the Riemannian Hessian. 
  
  If $Q^*$ is the dual connection of $Q$, then \eqref{eq:NewtonMinDer} and \eqref{eq:NewtonMinGrad} yield the same Newton steps. However, \eqref{eq:NewtonMinDer} does not require any Riemannian structure on $\X$. In particular the application of $\mathcal R_x^{-1}$ for the computation of the gradient can be dispensed with. This is a practical advantage if, for example, $\X$ is a Sobolev manifold, where $\mathcal R_x^{-1}$ is expensive to evaluate. 
\end{remark}

%
%
\section{Affine covariant damping}%
\label{sec:AffineCovariantDamping}%
In~\cite[Sec. 6]{WeiglSchiela:2024}, a damped version of Newton's method was derived. Using an affine covariant damping strategy in the spirit of ~\cite{Deuflhard} this paper introduced a geometrically motivated strategy based on following so-called \emph{Newton paths}. For the sake of completeness and clarity, we recall the main ideas.

Consider a mapping $F\colon\X\to\E$ into a vector bundle $p\colon\E\to\Y$. Denote by $x(0)\in \X$ the starting point of the path. For $\alpha \in [0,1]$ and $x(\alpha)\in\X$ the base point of $F(x(\alpha)) \in \E$ is denoted by $y(\alpha)\coloneq p(F(x(\alpha)))\in \Y$. Let $\VTback{y} \in \Gamma(\L(\E,E_y))$ be a vector back-transport.
Then the \textit{Newton path problem}, which is based on the idea of scaling down the residual by a factor of $1-\alpha$, is given as
\begin{equation}
    \label{NewtonPath}
    \NE{\VTback{{y(0)}}}(F(x(\alpha))) = (1-\alpha) F(x(0)), \; \alpha \in [0,1].
\end{equation}
Since the residuals $F(x(\alpha))$ and $F(x(0))$ do not lie in the same fibre, we need a vector back-transport on the left-hand side. This vector back-transport allows us to formulate the Newton path problem on a linear space, namely in the fixed fibre $E_{y(0)}$. The mapping $\alpha \mapsto x(\alpha)$ (where it is defined) is called the \textit{algebraic Newton path} starting at $x(0)$. Then, the idea is to scale the tangential direction at $x(0)$ by a suitable
damping factor $\alpha\in (0,1]$ in order to follow the
Newton path. For our algorithm it is crucial that the vector back-transport $\VTback{y}$ used to define a Newton path and the connection map $Q_e$ at $e$ with $p(e)=y$ used to compute Newton directions are consistent, i.\,e.\ $Q_e = \NE{\VTback{y}}'(e)$. For this choice it can be shown that Newton directions are tangential to the Newton paths (cf.~\cite[Sec. 6]{WeiglSchiela:2024}). \\
\newline
In the following we will work out the quantities necessary for the algorithm in the case of dual vector bundles and include an adapted algorithm. We will see that for dual bundles the necessary back-transports of the covectors $l\in \E^*$ can be constructed from adjoints of vector forward transports in the primal bundle. With this, we can derive consistent dual connection maps by differentiation.

\subsection{Dual connections induced by covector back-transports.}
Let $p\colon\E\to\Y$ be a vector bundle and $p^*\colon\E^* \to \Y$ its dual bundle.
Consider a vector transport $\VTforward{y} \in \Gamma(\L(\Y \times E_{y}, \E))$.
We can define a \emph{covector back-transport} $\VTstar{y} \in \Gamma(\L(\E^*, E_{y}^*))$, which acts as a vector back-transport on the dual bundle by using fibrewise adjoints $\VTstar{y}(\hat{y})\coloneqq(\VTforward{y}(\hat{y}))^* \in L(E_{\hat{y}}^*,E_{y}^*)$, i.e.
\begin{equation*}
 \VTstar{y}(\hat{y})\hat l \coloneqq \hat l \circ \VTforward{y}(\hat{y}) \quad \text{ for all } \hat l \in E_{\hat{y}}^*.
\end{equation*}
Then, we can define a connection map on the dual bundle $p^*\colon\E^* \to \Y$ via this covector back-transport by differentiation: For $l \in \E^*$ we set
\begin{equation}\label{eq:dualconnection}
 Q_{l}^*\coloneqq \NE{\VTstar{y}}'(l):T_{l}\E^*  \to E^*_{y}, \;  y=p^*(l).
\end{equation}
In the following we will show that this defines a dual connection map on $\E^*$. To prove this, we first define an appropriate connection map on the primal bundle $\E$ and then show that the product rule~\eqref{eq:defdualconnection} is satisfied for these choices.

Using the vector transport $\VTforward{y}$ we can locally define a vector back-transport on the vector bundle $\E$ by taking fibrewise inverses, i.\,e.\ we set
\begin{equation*}
    \VTback{y}(\hat{y}) \coloneqq \left(\VTforward{y}(\hat{y})\right)^{-1} \colon E_{\hat{y}} \to E_y.
\end{equation*}
Then, we can define a connection map on the primal bundle $p\colon\E \to \Y$ by setting $Q_e = \NE{\VTback{y}}'(e)$ for $e\in \E$ with $y=p(e)$.
For the sake of clarity, we repeat a result on the representation of $Q_e$ in trivializations (cf.~\cite[Lemma 2.2]{WeiglSchiela:2024}).
\begin{lemma}
\label{lem:VTconnection}
Let $\VTforward{y}\in \Gamma(\L(\Y\times E_y,\E))$ be a vector transport and $\VTback{y}$ be a vector back-transport defined via fibrewise inverses. Then,
\begin{equation*}
    Q_{e}\coloneqq \NE{\VTback{y}}'(e):T_{e}\E \to E_{y}
\end{equation*}
defines a linear connection map at $e\in \E$, $y= p(e)$, which is represented in trivializations by~\eqref{eq:Qtriv}  with
\begin{equation}\label{eq:diffVtTriv}
B_{y,\tau}(e_\tau)\delta y = -(\VTprime{{y,\tau}}{\leftarrow}(y)\delta y) e_\tau = (\VTprime{{y,\tau}}{\rightarrow}(y)\delta y) e_\tau.
\end{equation}
\end{lemma}
Using this, we can show that~\eqref{eq:dualconnection} defines a corresponding dual connection map, i.\,e.\ $Q_e$ and $Q_l^*$ satisfy the product rule.
\begin{proposition}\label{pro:propdual}
    Let $\VTforward{y}\in \Gamma(\L(\Y\times E_y,\E))$ be a vector transport. The choice~\eqref{eq:dualconnection} defines the dual connection map $Q_l^*$ on $\E^*$ corresponding to the connection map $Q_e = \NE{\VTback{y}}'(e)$ at $e$ with $y=p(e)=p^*(l)$ on $\E$ where $\VTback{y}$ is defined by taking fibrewise inverses.
\end{proposition}
\begin{proof}
    Let $\ell \in \Gamma(\E^*)$, $e\in\Gamma(\E)$, and $\delta y\in T_y\Y$. We will show that the product rule~\eqref{eq:defdualconnection} is satisfied for $Q_{e(y)} = \NE{\VTback{y}}^\prime(e(y))$ and $Q_{\ell(y)}^* = \NE{\VTstar{y}}'(\ell(y))$. In trivializations, if $e(y)$ is represented by $(y, e_\tau(y))$, $\ell(y)$ by $(y, \ell_\tau(y))$ and $\ell'(y)\delta y$ by $(\delta y, \delta \ell_\tau)$, we obtain for the left hand side of~\eqref{eq:defdualconnection} by using $\NE{\VTstar{y}}(\ell(\hat y)) = \ell(\hat y) \circ \VTforward{y}(\hat y)$ and~\eqref{trivializationDerivative}:
    \begin{equation*}
        \left(\NE{\VTstar{y}}'(\ell(y))\ell'(y)\delta y\right)(e(y)) = \delta \ell_\tau(e_\tau(y)) + \ell_\tau(y) \circ \VTforward{y,\tau}'(y)\delta y(e_\tau(y)).
    \end{equation*}
    The right hand side of~\eqref{eq:defdualconnection} has the representation
    \begin{align*}
 \ell(e)'(y) \delta y - \ell(y)( Q_{e(y)}\circ e'(y) \delta y)&=\ell_\tau(e_\tau)'(y)  - \ell_\tau(y)(\delta e_\tau-B_{y,\tau}(e_\tau(y))\delta y)\\
 &=\delta \ell_\tau(e_\tau(y))+\ell_\tau(y)(B_{y,\tau}(e_\tau(y))\delta y).
\end{align*}
    Since $B_{y,\tau}(e_\tau(y))\delta y = -(\VTback{y,\tau}'(y)\delta y)e_\tau(y) = (\VTforward{y,\tau}'(y)\delta y)e_\tau(y)$ by~\eqref{eq:diffVtTriv} these terms coincide in all trivializations and thus, we obtain the desired result.
\end{proof}
The following proposition describes a practical way how to apply a dual connection to $F' \in \Gamma(\L(T\X, F^*T\E))$. It will be useful for us in Section~\ref{sec:applications} below.
\begin{proposition}
    \label{prop:computationDualConnection}
    Let $\VTforward{y}\in \Gamma(\L(\Y\times E_y,\E))$ be a vector transport. Consider the dual connection $Q^*$ given by \eqref{eq:dualconnection}.  Let $F \in C^1(\X,\E^*)$. Then for all $x\in\X$ and $e\in E_{y(x)}$ the following holds:
   \begin{equation*}
    (Q^*_{F(x)}\circ F'(x)\delta x)(e)=\frac{d}{d\xi}\Big(F(\xi)(\VTforward{y(x)}(y(\xi))e)\Big)\Big|_{\xi=x}\delta x \qquad \text{ for all } \delta x \in T_x\X.
   \end{equation*}
\end{proposition}
\begin{proof}
        Let $x\in \X$ and $e\in E_{y(x)}$. We compute, using local trivializations $\tau$ of $\E$,~\eqref{eq:leTriv}, the product rule and~\eqref{eq:NewtonMapTriv} with $B_{y(x),\tau}(e_\tau)y'(x)\delta x = \VTforward{y(x),\tau}'(y(x))y'(x)\delta x(e_\tau)$ for arbitrary $\delta x\in T_x\X$:
\begin{align*}
 \frac{d}{d\xi}\Big(F(\xi)(\VTforward{y(x)}(y(\xi))e)\Big)\Big|_{\xi=x}\delta x &= \frac{d}{d\xi}\Big(F_\tau(\xi)(\VTforward{y(x),\tau}(y(\xi))e_\tau)\Big)\Big|_{\xi=x}\delta x\\
 &=(F'(x)\delta x)_\tau(e_\tau)+(F_\tau(x)\circ \VTforward{y(x),\tau}'(y(x))y'(x)\delta x)(e_\tau)\\
 &=Q^*_{F(x),\tau}(y'(x)\delta x, (F'(x)\delta x)_\tau)(e_\tau) \\
 &= (Q^*_{F(x)}\circ F'(x)\delta x)(e).
\end{align*}    
\end{proof}
In accordance with Def. \ref{def:consistentConnectionPrimal} we call a covector back-transport $\VTstar{y}\in \Gamma(\L(\E^*, E_y^*))$  \emph{consistent} with a connection map $Q^* \in \Gamma(\L(T\E^*, (p^*)^*\E^*))$ at $l \in \E^*$ with $y=p^*(l)$, if $Q^*_l = \NE{\VTstar{y}}^\prime(l)$.

\subsection{An affine covariant damped Newton method}
In the following we assume that the dual connection map $Q^*$ on $\E^*$ is consistent with a covector back-transport $\VTstar{y} \in \Gamma(\L(\E^*, E_y^*))$, i.\,e.\ $ Q_{l}^*\coloneqq \NE{\VTstar{y}}^\prime(l)$ for $l\in \E^*$.

We employ a damping factor $\alpha \in (0,1]$ to compute a damped Newton update in the following way:
\begin{equation*}
  x_+=R_x(\alpha \delta x).
\end{equation*}
To assess, whether $x_+$ is acceptable, we use the affine covariant quantity
\begin{equation*}
    \theta_{x_+}(x) = \cfrac{\lVert\overline{\delta x_+^\alpha}\rVert_x}{\lVert\alpha \delta x\rVert_x}
\end{equation*}
which contains computing the simplified Newton direction $\overline{\delta x_+^\alpha} \in T_x\X$ (for a detailed motivation we refer to~\cite{WeiglSchiela:2024} and~\cite{Deuflhard}). Using the consistency of the covector back-transport and the dual connection map, this direction can be computed by solving the linear equation
\begin{equation}
    \label{eq:SimplifiedNewtonDual}
    Q^*_{F(x)} \circ F'(x)\overline{\delta x_+^\alpha} + \VTstar{y(x)}(y(x_+))F(x_+) - (1-\alpha)F(x) = 0_{y(x)}^*.
\end{equation}
In particular, the covector back-transport needed here can be performed via a forward transport of the test-vectors:
\begin{equation*}
 \VTstar{y(x)}(y(x_+))F(x_+)(e)=F(x_+)(\VTforward{y(x)}(y(x_+))e) \quad \text{ for all } e\in E_{y(x)}.
\end{equation*}
Then~\eqref{eq:SimplifiedNewtonDual} can be stated equivalently as
\begin{equation*}
    Q^*_{F(x)} \circ F'(x)\overline{\delta x_+^\alpha}e + F(x_+)(\VTforward{y(x)}(y(x_+))e) - (1-\alpha)F(x)e = 0 \quad \text{ for all } e\in E_{y(x)}.
\end{equation*}
\begin{remark}
By Proposition~\ref{prop:computationDualConnection} we see that $Q^*_{F(x)} \circ F'(x)$ can be computed by differentiating these vector transports of the test vectors at $x$ to obtain a consistent connection. This is of central practical importance for the numerical computations, presented below. We also stress that our requirement of consistency gives us a lot of flexibility for numerical implementations, since we are not confined to a specific, say Levi-Civita, connection. Instead, we are free to choose a 
vector transport, and define our connection accordingly. 
\end{remark}

The resulting algorithm can then be stated as summarized in Algorithm~\ref{GlobalerNewton}.
\begin{algorithm}[h]
    \caption{Affine covariant damped Newton's method}
    \label{GlobalerNewton} 
    \begin{algorithmic}[1]
    \REQUIRE $x, \; \alpha \text{ (initial guesses)}; \; \alpha_{fail}, \; TOL$, $\Theta_{des} < \Theta_{acc}$ (parameters)
    \REPEAT
        \STATE solve $\delta x \leftarrow (Q^*_{F(x)} \circ F^\prime (x))\delta x + F(x) = 0^*_{y(x)}$
            \REPEAT
                \STATE compute $x_+ = R_x(\alpha \delta x)$
                \STATE solve $\overline{\delta x^\alpha_+} \leftarrow  Q^*_{F(x)} \circ F^\prime (x) \overline{\delta x^\alpha_+} + \VTstar{y(x)}(y(x_+))F(x_+) - (1-\alpha)F(x) = 0_{y(x)}^*$
                \STATE compute $\theta_{x_+}(x) = \cfrac{\lVert\overline{\delta x_+^\alpha}\rVert_x}{\lVert\alpha \delta x\rVert_x}$
                \STATE update $\alpha \leftarrow \min \left(1, \cfrac{\alpha\Theta_{des}}{\theta_{x_+}(x)}\right)$
                 \IF{$\alpha < \alpha_{fail}$}
                       \STATE \textbf{terminate:} "Newton's method failed"
                \ENDIF
            \UNTIL{$\theta_{x_+}(x) \leq \Theta_{acc}$}
        \STATE  update $x \leftarrow x_+$
        \IF{$\alpha = 1 \text{ and } \theta_{x_+}(x) \leq \frac{1}{4} \text{ and } \lVert\delta x\rVert_x \leq TOL$}
            \STATE \textbf{terminate:} "Desired Accuracy reached", $x_{out} = x_+$
        \ENDIF
    \UNTIL{maximum number of iterations is reached}
    \end{algorithmic}
\end{algorithm}
%
%
\section{Connections for mappings into embedded vector bundles}%
\label{sec:CovDerivEmbedded}%
In our numerical examples we will consider mappings into vector bundles which are embedded into Hilbert spaces. In view of our globalization strategy, we need primal and dual connections that are consistent with vector transports. Thus, we will now discuss an easy way to derive vector transports and corresponding connection maps in this case by using orthogonal projections onto the fibres.\\
\newline
Let $p\colon \E \to \Y$ be a vector bundle that is embedded into a Hilbert space $\Hilbert$. Then we can use the embedding $\iota_{\Hilbert}\colon \E \to \Hilbert$ with $\iota_{\Hilbert}(y)\colon E_{y}\to \Hilbert$ for $y\in \Y$ and the orthogonal projection $P(y)\colon \Hilbert \to E_y$ onto the fibre to define a vector back-transport locally by
\begin{align*}
    \VTback{y}(\eta) \coloneqq P(y)\iota_{\Hilbert}(\eta) \in L(E_{\eta},E_y).
\end{align*}
Analogously, we can define a vector transport locally by
\begin{equation*}
    \VTforward{y}(\eta) \coloneqq P(\eta)\iota_{\Hilbert}(y) \in L(E_y, E_{\eta}).
\end{equation*}
A covector back-transport is then given by taking fibrewise adjoints $\VTstar{y}({\eta}) \coloneqq (\VTforward{y}({\eta}))^*$, i.e.
\begin{equation}
    \label{eq:CovectorTransportProjection}
    \VTstar{y}({\eta})l = l \circ \VTforward{y}({\eta})=l \circ P({\eta})\circ \iota_{\Hilbert}(y) \quad \text{ for all } l\in E_{\eta}^*.
\end{equation}
As seen before, we may obtain consistent connection maps by differentiation.
\subsection*{Primal connection maps} Using the natural inclusion $\varphi$ given by \eqref{eq:defNaturalInclusion} we obtain the $C^1$-mapping $\NE{\VTback{y}} \colon \E \to E_y$. Let $e \in \Gamma(\E)$. Since $p(e(\eta)) = \eta$, we get $\NE{\VTback{y}}(e(\eta)) = P(y)\iota_{\Hilbert}(\eta)e(\eta)$. Differentiating this with respect to ${\eta}$ at ${\eta} = y$, we obtain
\begin{equation}
    \label{eq:EmbeddedQPrimal}
    Q_{e(y)}\circ e'(y)\delta y=\NE{\VTback{y}}'(e(y))e'(y)\delta y = P(y)(\iota_{\Hilbert}e)'(y)\delta y \in E_y.
\end{equation}

\begin{remark}
For a vector field $\nu\colon \X \to T\X$ we obtain by using~\eqref{eq:EmbeddedQPrimal} and $\pi(\nu(x)) = x$ for $x\in \X$
\begin{equation*}
    Q_{\nu(x)}\circ \nu'(x)\delta x=\NE{\VTback{x}}'(\nu(x))\nu'(x)\delta x = P(x)(\iota_{\Hilbert}\nu)'(x)\delta x \in T_x\X,
\end{equation*}
which is consistent with the orthogonal projection.
The \emph{covariant derivative} $Q\circ \nu'$ of a vector field $\nu$ w.r.t.\ the connection map $Q$ can therefore be determined by orthogonally projecting the derivatives of the embedded vector field onto $T\X$. The connection which belongs to this covariant derivative is sometimes called \textit{tangential connection} and defines the Levi-Civita connection on $T\X$, see e.g.~\cite[Chap. 5]{Lee:2018:1}.
\end{remark}

\subsection*{Dual connection maps} Consider $\ell_{\Hilbert}\colon \Y \to \Hilbert^*$ and its restriction $\ell \colon \Y \to \E^*$, given by $\ell({\eta})\coloneqq\ell_{\Hilbert}({\eta})\circ \iota_{\Hilbert}({\eta})$.
By the natural inclusion $\varphi$, we obtain the $C^1$-mapping $\NE{\VTstar{y}}(\ell({\eta})) = \ell_{\Hilbert}({\eta}) \circ \iota_{\Hilbert}({\eta})\circ \VTforward{y}({\eta})$.
Differentiating this with respect to ${\eta}$ at ${\eta} = y = p^*(\ell)$ yields
\begin{align*}
    \NE{\VTstar{y}}'(\ell(y))\ell'(y)\delta y &= \ell_{\Hilbert}'(y)\delta y + \ell_{\Hilbert}(y)(\iota_{\Hilbert}\circ \VTforward{y})'(y)\delta y\\
    &= \ell_{\Hilbert}'(y)\delta y + \ell_{\Hilbert}(y)(\iota_{\Hilbert}\circ P)'(y)\iota_{\Hilbert}(y)\delta y,
\end{align*}
or, suppressing the embeddings notationally:
\begin{align}
    \label{eq:Qstarembedded}
    Q^*_{\ell(y)}\circ \ell'(y)\delta y=\NE{\VTstar{y}}'(\ell(y))\ell'(y)\delta y = \ell_{\Hilbert}'(y)\delta y + \ell_{\Hilbert}(y)P'(y)\delta y \in E_y^*.
\end{align}
Vice versa, if $\ell \in \Gamma(\E^*)$ is given, we may extend $\ell$ to $\ell_H\colon \Y \to \Hilbert^*$, for example by specifying $\ell_{\Hilbert}(y)$ on $\iota_{\Hilbert}(E_y)^\perp$.
Then by its intrinsic definition $Q^*_{\ell(y)}\circ \ell'(y)\delta y$ is independent of the extension, and thus also~\eqref{eq:Qstarembedded} is independent of the extension.

\subsection*{Submersed vector bundles.}
Let us consider a smooth fibrewise linear mapping $A \colon \mathcal Y \to L(\Hilbert,\RH)$ and $A(y)$ is assumed to be surjective for each $y\in \mathcal Y$. Let $p\colon \E \to \Y$ be defined by $E_y=\mathrm{ker}\, A(y)$.
For a covector field $\ell_{\Hilbert}\colon \mathcal Y \to \Hilbert^*$ and its restriction $\ell\colon \mathcal Y \to \mathcal E^*$ we consider the problem:
\begin{equation*}
 0^*_y = \ell(y) \quad \Leftrightarrow \quad  0 = \ell_{\Hilbert}(y)e \quad \text{ for all } e \in E_y
\end{equation*}
In view of~\eqref{eq:Qstarembedded} we obtain for the connection, induced by the orthogonal projection $P(y)\colon {\Hilbert} \to E_y$:
\begin{equation}\label{eq:connorth}
 Q^*_{\ell(y)}\ell'(y)\delta y=\ell_{\Hilbert}'(y)\delta y+\ell_{\Hilbert}(y)P'(y)\delta y \quad \text{ for all } \delta y \in T_y\Y.
\end{equation}
Since $A(y)$ is surjective, the restriction $A_K(y)\colon\mathrm{ker}\, A(y)^\perp \to \RH$ is an isomorphism by the open mapping theorem, and so is its adjoint $A_K(y)^* \colon \RH^* \to (\mathrm{ker} A(y)^\perp)^*$.
This means that there is a unique \emph{Lagrangian multiplier}
\[
\lambda(y):=-(A_K(y)^*)^{-1}\ell_{\Hilbert}(y)\in \RH^*,
\]
satisfying
\begin{equation}\label{eq:Lagrange}
(\ell_{\Hilbert}(y)+\lambda(y) A(y))w=0 \quad \text{ for all } w\in \mathrm{ker}\, A(y)^\perp,
\end{equation}
and thus
\[
 \ell_{\Hilbert}(y)e=0 \quad \text{ for all } e\in \mathrm{ker}\,A(y) \quad \Leftrightarrow \quad (\ell_{\Hilbert}(y)+\lambda(y) A(y))v \quad \text{ for all } v \in {\Hilbert}^*.
\]
Using $\mathrm{ker}\, A(y)=\mathrm{ran}\,P(y)$ we can write \eqref{eq:Lagrange} as
\begin{equation*}
\begin{split}
 0&=(\ell_{\Hilbert}(y)+\lambda(y) A(y))(Id_{\Hilbert}-P(y))v\\
 &=\ell_{\Hilbert}(y)(Id_{\Hilbert}-P(y))v+\lambda(y) A(y)v \quad \text{ for all } v\in {\Hilbert}.
 \end{split}
\end{equation*}
Differentiation with respect to $y$ in direction $\delta y\in T_y\Y$ yields for all $v\in H$:
\begin{equation*}
 0=\Big((\ell'_{\Hilbert}(y)\delta y)(Id_{\Hilbert}-P(y))-\ell_{\Hilbert}(y)P'(y)\delta y+(\lambda'(y)\delta y)\,A(y)+\lambda(y) A'(y)\delta y\Big)v
\end{equation*}
and thus, testing with $e\in \mathrm{ker}\,A(y)$, so that $A(y)e=(Id_{\Hilbert}-P(y))e=0$, we get:
\begin{equation*}
 \lambda(y) A'(y)(\delta y,e)=\ell_{\Hilbert}(y)P'(y)\delta y\, e \quad \text{ for all } e\in \mathrm{ker}\, A(y).
\end{equation*}
Hence, in \eqref{eq:connorth} we can dispense with the use of $P'(y)$, if $\lambda(y)$ is available:
\begin{align}\label{eq:Qstarconstrained}
    Q^*_{\ell(y)}\ell'(y)\delta y = \ell'_{\Hilbert}(y)\delta y + \lambda(y) A'(y)\delta y\quad \text{ for all } \delta y \in T_y \Y.
\end{align}
\begin{remark}
 To cover constrained optimization problems of the form
 \begin{equation*}
 \min_{x\in {\Hilbert}} f(x) \text{ s.t. } c(x)=0
 \end{equation*}
 we may set $\mathcal Y=\{c(x)=0\}$, $\ell_{\Hilbert}(y)=f'(x)$, $\ell_{\Hilbert}'(y)=f''(x)$, $A(y)=c'(x)$ and $A'(y)=c''(x)$. Then~\eqref{eq:Qstarconstrained} yields the Hessian of the Lagrangian function
\begin{align}\label{eq:QstarconstrainedOpt}
    Q^*_{f'(x)}f''(x)\delta x = f''(x)\delta x + \lambda(x) c''(x)\delta x \quad \text{ for all } \delta x  \in \mathrm{ker}\, c'(x),
\end{align}
which occurs in the second order analysis and in Lagrange-Newton methods for constrained optimization. It can thus be interpreted as the covariant second derivative of $f$ on $\Y$.
\end{remark}

\section{Applications}
\label{sec:applications}
In this section we will discuss the computational aspects of solving variational problems on manifolds. Our main emphasis is on algorithmic questions, and in particular the realization of the abstract concepts, introduced so far, to a couple of concrete problems. We deliberately do not touch questions of (numerical) analysis, since this is beyond the scope of the current work. 
\subsection{Elastic Geodesics in a Force Field}
\label{sec:ElasticGeodesics}
A geodesic minimizes the length, or equivalently, the Dirichlet energy~\eqref{eq:Dirichlet} among all curves that connect two given points on a manifolds. An alternative interpretation of such a curve would be that of an elastic string, constrained to the manifold, which assumes the configuration of minimal energy. A straightforward generalization of this model is to add a force field to this problem. We end up with an \emph{elastic geodesic} in a force field.

Let $\gamma\colon I \to \M$ where $\M$ is an embedded submanifold of $\R^d$ with $\dim \M = m$ and $I \coloneqq [0,T] \subset \R$ for $T\in \R$.
We use the euclidean inner product $\langle \cdot, \cdot \rangle$ of $\R^d$ as a Riemannian metric on $\M$. Let $\X$ be a Sobolev manifold (cf. e.g. \cite{hajlasz2009sobolev})
\begin{equation*}
    \X = H^1(I, \M) \coloneqq \{ g \in H^1\left(I, \R^d\right) \; \mid \; g(t) \in \M \text{ a.e.} \}
\end{equation*}
and $\E = T\X$ its tangent bundle.
Let $\omega\colon \M \to T^*\M$ be a force field on $\M$. In general $\omega$ may have no anti-derivative (is not exact as a $1$-form), so in terms of physics, $\omega$ may be a non-conservative force field. 
This is in particular the case, if $\omega'$ is not symmetric, or equivalently, the outer derivative $d\omega$ is non-zero. 

Our goal is to find a zero of the following mapping:
\begin{align*}
    F\colon \X &\to \E^*\\
    \gamma &\mapsto F(\gamma), \quad F(\gamma)\phi \coloneqq \int_I \langle \dot{\gamma}(t), \dot{\phi}(t)\rangle + \omega(\gamma(t))\phi(t) \; dt \text{ for } \phi \in T_\gamma\X.
\end{align*}
Additionally, we have to take into account that boundary conditions $\gamma(0) = \gamma_0$ and $\gamma(T) = \gamma_T$ for given $\gamma_0, \gamma_T \in \M$ are satisfied. If $\omega \equiv 0$ this coincides with the variational formulation \eqref{eq:GeodesicsVariational} of the classical geodesic problem.
In case of a non-conservative force-field $\omega$ this problem can not be phrased as an energy minimization problem. We call a solution $\gamma$ of this mapping an \emph{elastic geodesic} in a force field $\omega$. Be aware that a solution of this problem is not guaranteed to exist in general. This depends on the choice of boundary points and force fields. 

In order to apply Newton's method to find a zero of $F$, we need the linear mapping $Q_{F(\gamma)}^*\circ F'(\gamma)$ which can be seen as a covariant derivative of a fibrewise linear mapping.

Here, we can use the embedding of $T\M$ into $\R^d$, define the covector back-transport according to~\eqref{eq:CovectorTransportProjection}, and derive the dual connection map applied to $F'$ by~\eqref{eq:Qstarembedded} using projections $P(\gamma(t))\colon \R^d \to T_{\gamma(t)}\M$, $t\in[0,T]$ in a pointwise fashion. To this end we apply the techniques of Section \ref{sec:CovDerivEmbedded} with $H=\R^d$ for each $t\in I$. Concretely, we write
\begin{equation*}
    F(\gamma)\phi = \int_I \langle \dot{\gamma}(t), (P(\gamma(t))\phi(t))\dot{} \;\rangle + \omega(\gamma(t))P(\gamma(t))\phi(t) \; dt
\end{equation*}
for $\phi\in T_\gamma\X$. Differentiating this with respect to $\gamma$ yields
\begin{equation}
\label{eq:covariantDerivativeGeodesics}
\begin{aligned}
    Q_{F(\gamma)}^*\circ F'(\gamma)\delta \gamma\,\phi &= \int_I \langle \dot{\delta \gamma}(t),\dot{\phi}(t)\rangle + \langle \dot{\gamma}(t), (P'(\gamma(t))\delta \gamma(t)\phi(t))\dot{} \; \rangle \\
    & \qquad + \omega'(\gamma(t))\delta \gamma(t)\phi(t) + \omega(\gamma(t))P'(\gamma(t))\delta \gamma(t)\phi(t) \; dt\\
    &=F(\gamma)(P'(\gamma)\delta \gamma\,\phi) + F_{\R^d}'(\gamma)\delta \gamma\,\phi
\end{aligned}
\end{equation}
where $F_{\R^d}'(\gamma)$ denotes the Euclidean derivative of $F$ at $\gamma$, i.e.
\begin{equation*}
    F_{\R^d}'(\gamma)\delta \gamma\, \phi = \int_I \langle \dot{\delta \gamma}(t),\dot{\phi}(t)\rangle + \omega'(\gamma(t))\delta \gamma(t)\phi(t) \; dt.
\end{equation*}
\begin{remark}
 Taking into account that $\X = H^1(I,\M)$ we observe that our pointwise defined connection is \emph{not} the dual of the Levi-Civita connection on $\X$. Due to the Sobolev structure, the Levi-Civita connection would be more difficult to derive and much more expensive to evaluate. The flexibility to employ an arbitrary vector transport together with a conistent connection admits the relatively simple derivation and an efficient computation of the Newton matrix here.
\end{remark}

\subsection*{Discretization} \label{sec:DiscretizationGeodesics}
To turn our infinite dimensional problem computationally tractable, we first have to discretize the right hand side $F(\gamma)$ and the Newton equation
\begin{equation*}
    Q_{F(\gamma)}^* \circ F'(\gamma)\delta \gamma + F(\gamma) = 0_\gamma^* \quad \Leftrightarrow \quad Q_{F(\gamma)}^* \circ F'(\gamma)\delta \gamma\, \phi  + F(\gamma)\phi = 0 \quad \text{ for all } \phi \in T_\gamma \X
\end{equation*}
to compute a Newton direction $\delta \gamma$. For simplicity we choose a continuous piecewise linear discretization of $\gamma$, which utilizes the embedding $\M \subset \R^d$. A discussion of more sophisticated intrinsic ways of discretization, like, for example,  geometric finite elements can be found in \cite{hardering2020geometric}. 

We consider equidistant discrete time points
\begin{equation*}
    t_0 = 0 < t_1 < \dots < t_{N+1} = T
\end{equation*}
with distance $h$, i.\,e.\ $t_i = ih, \; i = 0,...,N+1$, such that
\begin{equation*}
    I = [0,T] = \bigcup_{i=0}^{N} [t_i, t_{i+1}].
\end{equation*}
The set of admissible curves is then defined by the piecewise linear interpolants in $\R^d$ of the points $y_i \in \M$, i.e.
\begin{equation*}
    \gamma(t_i) =  y_i \in \M, \; i=1,..,N.
\end{equation*}
Thus, $\gamma$ is represented by a point $y = (y_1, \cdots y_N)$ on the \emph{product manifold} $\M^N$. Additionally, the points $y_0$ and $y_{N+1}$ are given by the boundary conditions $\gamma(0) = \gamma_0$ and $\gamma(T) = \gamma_T$. Setting
\begin{equation*}
    w_i \coloneqq \omega(\gamma(t_i)) = \omega(y_i) \in T_{\gamma(t_i)}\M^* \subset (\R^d)^* \text{ for } i = 0,\dots, N,
\end{equation*}
yields $w = (w_0, \dots, w_{N+1}) \in ((\R^d)^*)^{N+2}$. Similarly, each test function $\phi$ is given as a piecewise linear interpolant in the following way
\begin{equation*}
(\phi(t_0), \dots, \phi(t_{N+1})) = (\delta y_0, \dots, \delta y_{N+1})  \in T_{y_0}\M \times \dots \times T_{y_{N+1}}\M
\end{equation*}
where $\delta y_i = \delta \gamma(t_i)$ for $i=1, \hdots, N$ and $\delta y_0 = \delta y_{N+1} = 0$. To construct a basis of the set of test functions choose a basis $\{v_{i,j}\}_{j=1\dots m}$ of each tangent space $T_{y_i}\M$ for $i=1\dots N$. Then for each $i$, $j$ we get a set of basis functions as follows:
\begin{equation}\label{eq:orthbasis}
 \phi_{(i-1)\cdot m + j}(t_k)=
 \begin{cases}
  v_{i,j} & \text{ if }i=k,\\
  0 & \text{ if } i \neq k,
 \end{cases}
 \quad i=1\dots N, j=1\dots m.
\end{equation}
The support of $\phi_{(i-1)\cdot m + j}$ consists of the interval $[t_{i-1},t_{i+1}]$. 
\subsection*{Assembly of the Newton system}
To compute the matrix and the right hand side for the Newton system we have to evaluate $F(\gamma)\phi$ for all basis functions $\phi$ and $Q_{F(\gamma)}^*\circ F'(\gamma)\phi^2\phi^1$ for all pairs of basis functions. More specifically, for $M= N \cdot m$ we have to find $b\in \R^{M}$ and $A\in \R^{M\times M}$ by computing
\begin{equation}\label{eq:NewtonAb}
 b_{k}\coloneqq F(\gamma)\phi_{k}, \qquad A_{k,l}\coloneqq Q_{F(\gamma)}^*\circ F'(\gamma)\phi_{l}\phi_{k} \quad k,l=1\dots M.
\end{equation}
Then, solving the equation $A\xi +b=0$ yields a representation of the Newton direction $\delta \gamma$ with respect to the chosen bases:
\begin{equation*}
 \delta \gamma(t) = \sum_{l=1}^M \xi_l \phi_l(t).
\end{equation*}
Let us first consider the right hand side
\begin{align*}
    F(\gamma)\phi
    &= \sum_{i=0}^N \int_{[t_i, t_{i+1}]} \left\langle \dot{\gamma}(t), \dot{\phi}(t)\right\rangle + \omega(\gamma(t))\phi(t) \; dt.
\end{align*}
Due to our hypothesis of linear interpolation, in every time interval $[t_i, t_{i+1}]$, the time derivatives of $\gamma$ and $\phi$ are constant and we can compute them by finite differences in the embedding, i.\,e.\ we use
\begin{equation*}
    \dot{\gamma}(t) = \frac{y_{i+1} - y_{i}}{h} \quad \text{and} \quad \dot{\phi}(t) = \frac{\delta y_{i+1} - \delta y_{i}}{h} \quad \text{for } t \in [t_i,t_{i+1}].
\end{equation*}
Thus, we get
\begin{equation*}
    F(\gamma)\phi = \sum_{i=0}^N \int_{[t_i, t_{i+1}]} \left\langle \frac{y_{i+1} - y_{i}}{h}, \frac{\delta y_{i+1} - \delta y_{i}}{h} \right\rangle + \omega(\gamma(t))\phi(t) \; dt
\end{equation*}
Using the trapezoidal rule to approximate the integral, we obtain the discretized right hand side by
\begin{equation}
    \label{eq:DiscretizationRHS}
    F(\gamma)\phi \approx h \cdot \sum_{i=0}^N \left\langle \frac{y_{i+1} - y_{i}}{h}, \frac{\delta y_{i+1} - \delta y_{i}}{h} \right\rangle + \frac{w_i\delta y_i + w_{i+1}\delta y_{i+1}}{2}.
\end{equation}
As we have seen above, $Q_{F(x)}^*\circ F'(x)$ applied to $\phi^1, \phi^2 \in T_{\gamma}\X$ can be written as
\begin{align*}
    Q_{F(\gamma)}^*\circ F'(\gamma)\phi^2\phi^1
    &= F(\gamma)(P'(\gamma)\phi^2\phi^1) + F_{\R^d}'(\gamma)\phi^2 \phi^1
\end{align*}
where $F_{\R^d}'(\gamma)$ is given by
\begin{align*}
    F_{\R^d}'(\gamma)\phi^2 \phi^1\\
    &= \sum_{i=0}^N \int_{[t_i, t_{i+1}]} \left\langle \dot{\phi^2}(t),\dot{\phi^1}(t)\right\rangle + \omega'(\gamma(t))\phi^2(t)\phi^1(t) \; dt.
\end{align*}
Again, setting $\delta y^1 = (\delta y^1_1, \dots, \delta y^1_N) \coloneqq (\phi^1(t_1), \dots, \phi^1(t_N))$, $\delta y^2 = (\delta y_1^2, \dots, \delta y^2_N) \coloneqq (\phi^2(t_1), \dots, \phi^2(t_N))$ and $\delta y^1_0 = \delta y^1_{N+1} = \delta y^2_0 = \delta y_{N+1}^2 = 0$ and using the trapezoidal rule yields an approximation
\begin{equation}
\label{eq:DiscretizedEuclideanDerivative}
\begin{aligned}
F_{\R^d}'(\gamma)\phi^2 \phi^1 &\approx h \cdot \sum_{i=0}^N \left\langle \frac{\delta y^2_{i+1} - \delta y^2_{i}}{h}, \frac{\delta y^1_{i+1} - \delta y^1_{i}}{h} \right\rangle \\
& \quad + h \cdot \sum_{i=0}^N \frac{w'(y_{i+1})\delta y^2_{i+1}\delta y^1_{i+1} + w'(y_{i})\delta y^2_{i}\delta y^1_{i}}{2}.
\end{aligned}
\end{equation}
We set
\begin{equation*}
    P_i \coloneqq P'(y_i)\delta y_i^2\delta y_i^1 \;\text{ for } \;i=0,\dots, N+1.
\end{equation*}
Thus, combining~\eqref{eq:DiscretizedEuclideanDerivative} with~\eqref{eq:DiscretizationRHS} applied to $P=(P_0,...,P_{N+1})$ yields an approximation
\begin{equation}\label{eq:DiscretizationMatrix}
\begin{split}
    Q_{F(\gamma)}^*\circ F'(\gamma)\phi^2\phi^1 &\approx h \cdot \sum_{i=0}^N \left\langle \frac{y_{i+1} - y_{i}}{h}, \frac{P_{i+1} - P_{i}}{h} \right\rangle + \frac{w_iP_i + w_{i+1}P_{i+1}}{2} \\
    &\quad + h \cdot \sum_{i=0}^N \left\langle \frac{\delta y^2_{i+1} - \delta y^2_{i}}{h}, \frac{\delta y^1_{i+1} - \delta y^1_{i}}{h} \right\rangle \\
    &\quad + h \cdot \sum_{i=0}^N \frac{w'(y_{i+1})\delta y^2_{i+1}\delta y^1_{i+1} + w'(y_{i})\delta y^2_{i}\delta y^1_{i}}{2}.
\end{split}
\end{equation}
Observe that most of the terms in these sum are zero, if the $\phi^i$ are basis functions, as described above, since their values are zero at most of the grid points. In our case, this implies that $A$ is a block tridiagonal matrix with block of size $m\times m$.
\subsection*{Implementation}

The algorithm~\ref{GlobalerNewton} is implemented in \mintinline{julia}|Julia|~\cite{BezansonEdelanKarpinskiViral:2017}
as an algorithm within the framework of \mintinline{julia}|Manopt.jl|~\cite{Bergmann:2022:1}.
Since this framework is built upon \mintinline{julia}|ManifoldsBase.jl| the algorithm can directly
work on arbitrary \mintinline{julia}|VectorBundle|s, especially on a \mintinline{julia}|TangentBundle|
from this interface for Riemannian manifolds.
Furthermore, this generic vector bundle implementation is available for any manifold from
\mintinline{julia}|Manifolds.jl|~\cite{AxenBaranBergmannRzecki:2023}.
For the discretized models this leads to employing the generic \mintinline{julia}|PowerManifold|
implementation to phrase the problem on $\mathcal M^N$.

The data~\eqref{eq:NewtonAb} for the Newton system is computed in the spirit of a finite element assembly routine by using a loop over the intervals $[t_{i},t_{i+1}]$ with indices $i=0\dots N$ and adding to $A$ and $b$ all contributions to~\eqref{eq:DiscretizationRHS} and~\eqref{eq:DiscretizationMatrix} of (pairs of) basis function $\phi_{(i-1)\cdot m+j}$, defined in~\eqref{eq:orthbasis}, whose supports meet $(t_{i},t_{i+1})$.

\subsection*{Numerical results}
We consider $\M=\mathbb{S}^2\subset \R^3$ equipped with the Riemannian metric $\langle \cdot, \cdot \rangle$ given by the Euclidean inner product restricted to each embedded tangent space with corresponding norm $\lVert \cdot\rVert$ and a closed interval $I \subset \R$. To monitor convergence we use an $\infty$-norm over these norms:
\[
 \|\delta \gamma\|_{\gamma,\infty}=\max_{i=1\dots N}\|\delta \gamma(t_i)\|,
\]
which is obviously not a norm of Riemannian type, but fits into the theoretical framework of \cite{WeiglSchiela:2024}. In contrast to the $2$-norm its size is largely independent of the number of discretization points. 

Let $\omega\colon \M \to T^*\M$ be the 1-form given by a (scaled) winding field, i.\,e.\ for $y\in \mathbb{S}^2$ we set
\begin{equation*}
    \omega(y) \coloneqq \frac{3 y_3}{y_1^2+y_2^2} \bigg\langle \begin{pmatrix}
        -y_2 \\ y_1 \\ 0
    \end{pmatrix}, \cdot \bigg\rangle \in (T_y\mathbb{S}^2)^*.
\end{equation*}

Note that $\omega$ is then a non-conservative force. For almost antipodal boundary points $\gamma_0$ and $\gamma_T$ we choose the connecting geodesic as initial guess for Newton's method.

For $y\in \mathbb{S}^2$ the orthogonal projection onto the tangent space $T_y \mathbb{S}^2 \cong y^\perp$ (which is needed for computing $Q_{F(\gamma)}^*\circ F'(\gamma)$ according to~\eqref{eq:covariantDerivativeGeodesics}) is given by
\begin{equation*}
    P(y)\colon \R^3 \to T_y\mathbb{S}^2, \quad P(y) \coloneqq Id - y\langle y, \cdot \rangle
\end{equation*}
where $Id$ denotes the identity map. The derivative of the orthogonal projection at $y\in\mathbb{S}^2$ in the direction $v \in T_y\mathbb{S}^2$ applied to $u\in T_y\mathbb{S}^2$ is given by
\begin{equation*}
    (P'(y)v)u = -y\langle v, u\rangle - v \langle y, u \rangle.
\end{equation*}
As a retraction on $H^1(I, \M)$ we use a pointwise retraction on the sphere given by the normalization, i.\,e.\ for $\gamma \in H^1(I, \M)$, $\delta \gamma \in T_{\gamma}H^1(I, \M)$ and for each $t\in I$ we set
\begin{equation}
    \label{eq:pointwiseRetraction}
    R_{\gamma(t)}(\delta \gamma(t)) \coloneqq \frac{\gamma(t)+\delta \gamma(t)}{\lVert\gamma(t)+\delta \gamma(t)\rVert} \in \mathbb{S}^2.
\end{equation}
Figure~\ref{fig:geodesicForce} shows the local superlinear convergence that we expected to see due to~\cite[Prop. 5.10]{WeiglSchiela:2024}. In particular, the local convergence does not seem to depend on the number of discretization points. In addition, the resulting elastic geodesic in the force field can be seen, as well as the initial geodesic and the force acting on each point of the curve.

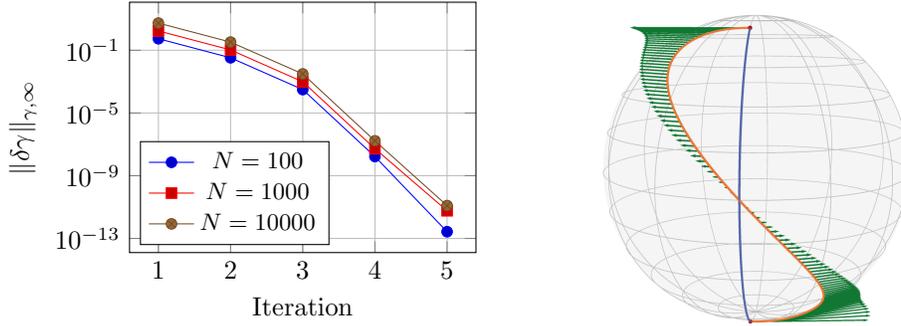
\begin{figure}[tbp]
\begin{minipage}[t]{0.54\textwidth}
    \begin{tikzpicture}
  \begin{axis}[
    xlabel={Iteration},
    ylabel={$\Vert \delta \gamma\Vert_{\gamma, \infty}$},
    ymode=log,
    grid=both,
    width=7cm,
    height=5.5cm,
    scale=0.85,
    xtick={1,2,3,4,5},
    ytick distance=10^4,
    legend pos=south west,
    legend style={font=\small}
  ]
        \addplot table [
      col sep=comma,
      header=false,
      x expr=\coordindex+1,   
      y index=0             
    ]{data/norm_newton_direction_geodesic_force_N100.csv};
    \addlegendentry{$N=100$}
     \addplot table [
      col sep=comma,
      header=false,
      x expr=\coordindex+1,   
      y index=0           
    ]{data/norm_newton_direction_geodesic_force_N1000.csv};
    \addlegendentry{$N=1000$}
     \addplot table [
      col sep=comma,
      header=false,
      x expr=\coordindex+1,   
      y index=0             
    ]{data/norm_newton_direction_geodesic_force_N10000.csv};
    \addlegendentry{$N=10000$}
  \end{axis}
\end{tikzpicture}
\end{minipage}
\begin{minipage}[t]{0.45\textwidth}
    \centering
    \hspace{0.5cm}\includegraphics[scale=0.5]{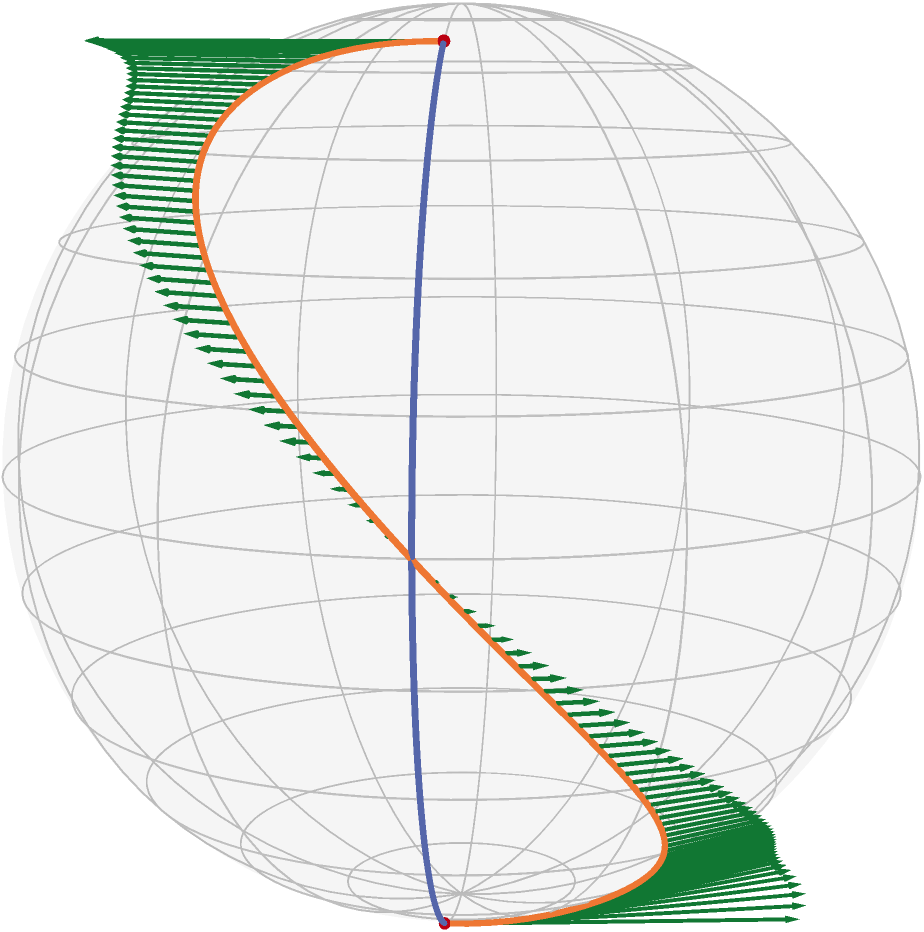}
\end{minipage}
\caption{Superlinear convergence of Newton's method for different numbers of discretization points $N=100, \, 1000, \,10000$ (left). Initial geodesic (blue) and resulting geodesic (orange) in the force field (green) for $N=100$ (right).}
\label{fig:geodesicForce}
\end{figure}

\subsection{An Obstacle Problem for Elastic geodesics}
\label{sec:ObstacleProblemElasticGeodesics}
We consider obstacle problems with elastic geodesics as a simple example for a variational inequality on a manifold. The idea is to find a curve of minimal Dirichlet energy that connects two given points, but avoids a prescribed area on the manifold. We will employ a penalty path-following method, where the sub-problems are not continuously differentiable, but only Newton-differentiable.
\newline

Let $\M$ be an embedded submanifold of $\R^d$ with $\dim \M = m$ and $I \coloneqq [0,T] \subset \R$ for $T\in\R$.  Consider the constrained minimization problem
\begin{equation}
    \label{eq:obstacle_constraint}
\begin{aligned}
    &\min_{\gamma \in H^1(I, \M)} \; \frac{1}{2} \int_I \lVert\dot \gamma(t)\rVert^2 \; dt \\
    & s.t. \; c(\gamma(t)) \leq 0 \; \text{ for all } t \in I
\end{aligned}
\end{equation}
where $c\colon\M \to \R^l$ is a twice differentiable mapping and the $\leq$ is meant pointwise. Additionally, we have the boundary conditions $\gamma(0) = \gamma_0$ and $\gamma(T) = \gamma_T$ for given $\gamma_0, \, \gamma_T \in \M$. Solving this problem yields a geodesic connecting $\gamma_0$ and $\gamma_T$ and avoiding an obstacle that is described by the given constraints. Using a penalty method, also known as Moreau-Yosida regularization (cf.~\cite{HintermuellerKunisch:2006:1}), with a quadratic penalty term we can rewrite this as an unconstrained minimization problem with a penalty coefficient $p\in \R$:
\begin{equation}
    \label{eq:obstacle_penalty}
    \begin{aligned}
        \min_{\gamma \in H^1(I, \M)} \; \frac{1}{2} \int_I \lVert\dot \gamma(t)\rVert^2 + p \sum_{i=1}^l \max(0, c_i(\gamma(t)))^2 \; dt
    \end{aligned}
\end{equation}
where $c_i\colon \M \to \R$ denotes the $i$-th component function of $c$.
We denote the objective of the penalized problem by
\begin{align*}
    f\colon H^1(I, \M) &\to \R \\
    \gamma &\mapsto f(\gamma) \coloneqq \frac{1}{2} \int_I \lVert\dot \gamma(t)\rVert^2 + p \sum_{i=1}^l m(c_i(\gamma(t)))^2 \; dt.
\end{align*}
Let $m \colon\R \to \R, \, m(x) \coloneqq \max(0, x)$. It can easily be shown that this function is Newton-differentiable with a Newton-derivative
\begin{equation}
    \label{eq:NewtonDerivativMax}
    m'(x) \coloneqq \begin{cases}
        0 &\text{ if }\; x<0 \\
        \text{arbitrary} &\text{ if }\; x= 0\\
        1 &\text{ if }\; x>0
    \end{cases}.
\end{equation}
By differentiating $f$ we obtain the mapping $f'\colon H^1(I, \M) \to  T^*H^1(I, \M)$ with
\begin{equation}
    \label{eq:firstderivativeEuclidean}
\begin{aligned}
   f'(\gamma)\phi &\coloneqq \int_I \langle \dot \gamma(t), \dot{\phi}(t)\rangle + p \sum_{i=1}^l m(c_i(\gamma(t)))m'(c_i(\gamma(t)))c_i'(\gamma(t))\phi(t)\; dt.
\end{aligned}
\end{equation}
For the maximum function it holds that
\begin{equation*}
    m(x)m'(x) = \begin{cases}
        0 & \text{ if }\; x \leq 0 \\
        m(x) & \text{ if }\; x > 0
    \end{cases} \; = m(x) \text{ for all } x\in \R.
\end{equation*}
Hence,~\eqref{eq:firstderivativeEuclidean} reduces to
\begin{equation*}
    f'(\gamma)\phi = \int_I \langle \dot \gamma(t), \dot{\phi}(t)\rangle + p \sum_{i=1}^l m(c_i(\gamma(t)))c_i'(\gamma(t))\phi(t)\; dt.
\end{equation*}
$f'$ is a semismooth mapping into a dual vector bundle, namely the cotangent bundle of $H^1(I, \M)$. Thus, we can apply Newton's method to compute a critical point of $f$. Defining the covector back-transport by orthogonal projections according to \eqref{eq:CovectorTransportProjection} and using the projection technique for embedded vector bundles derived in Sec.~\ref{sec:CovDerivEmbedded}, we obtain a dual connection map
\begin{equation*}
    Q_{f'(\gamma)}^*\circ f''(\gamma)\delta \gamma \, \phi = f_{\R^d}''(\gamma)\delta \gamma \, \phi + f'(\gamma)(P'(\gamma)\delta \gamma \, \phi).
\end{equation*}
Here, using a Newton-derivative of the maximum function according to \eqref{eq:NewtonDerivativMax}, we can write a Newton-derivative of $f'$, denoted by $f_{\R^d}''(\gamma)$, as
\begin{align*}
    f''_{\R^d}(\gamma)\delta \gamma \, \phi &= \int_I \langle \dot{\delta \gamma}(t), \dot{\phi}(t)\rangle + p \sum_{i=1}^l m'(c_i(\gamma(t)))c_i'(\gamma(t))\delta \gamma(t) c_i'(\gamma(t))\phi(t) \\
    & \quad + p \sum_{i=1}^l m(c_i(\gamma(t)))c_i''(\gamma(t))\delta \gamma(t)\phi(t) \; dt.
\end{align*}
\subsection*{Numerical results}
Consider $\M = \mathbb{S}^2 \subset \R^3$ and a closed interval $I \subset \R$. We want to compute a geodesic connecting two points $\gamma_0$ and $\gamma_T$ on the sphere while avoiding the north pole cap. For a given height $h_{\mathrm{ref}} \in (0,1)$ we consider the minimization problem
\begin{align*}
&\min_{\gamma\in H^1(I, \M)} \; \int_I \lVert\dot\gamma(t)\rVert^2 \; dt \\
&s.t. \; \gamma_3(t) \leq 1-h_{\mathrm{ref}} \; \text{ for all } t\in I
\end{align*}
where $\gamma_3(t)$ denotes the third component of $\gamma(t)\in \mathbb{S}^2$. Thus, we have one constraint given by $c\colon \mathbb{S}^2 \to \R, \; c(y) = y_3 - 1 + h_{\mathrm{ref}}$.
In this case a Newton-derivative according to~\eqref{eq:firstderivativeEuclidean} of the objective of the penalized problem given by~\eqref{eq:obstacle_penalty} reads
\begin{equation*}
    f'(\gamma)\phi = \int_I \langle \dot\gamma(t), \dot{\phi}(t)\rangle + p \cdot m(\gamma_3(t) - 1 + h_{\mathrm{ref}})\phi_3(t).
\end{equation*}
Since $c'(y) = (0,0,1)$ and $c''(y) = 0 \in \R^{3\times 3}$, the formula for the second Euclidean derivative reduces to
\begin{equation*}
    f''_{\R^3}(\gamma)\delta \gamma \,\phi = \int_I \langle \dot{\delta \gamma}(t), \dot{\phi}(t)\rangle + p \cdot m'(\gamma_3(t) - 1 + h_{\mathrm{ref}})\delta \gamma_{3}(t) \phi_{3}(t).
\end{equation*}
To compute a stationary point of $f$ by Newton's method, we have to discretize the Newton equation
\begin{equation*}
    Q^*_{f'(\gamma)} \circ f''(\gamma)\delta \gamma +f'(\gamma) = 0^*
\end{equation*}
to compute a Newton direction $\delta \gamma$. Here, we use the same approximation techniques as in the example before. The retraction which is necessary to define Newton steps is chosen as a pointwise retraction by normalization (cf.~\eqref{eq:pointwiseRetraction}).
For the computation of a solution of the penalized problem we use a simple path-following method increasing the penalty parameter by a factor 1.2 in each iteration. For the cap with $h_{\mathrm{ref}}=0.1$, respectively $h_{\mathrm{ref}} = 0.2$, this yields the elastic geodesics shown in Fig.~\ref{fig:geodesicObstacle} avoiding the north pole cap and connecting two points $\gamma_0$ and $\gamma_T$.%
\begin{figure}[tbp]
    \begin{minipage}[t]{0.49\textwidth}
        \centering
        \includegraphics[scale=0.5]{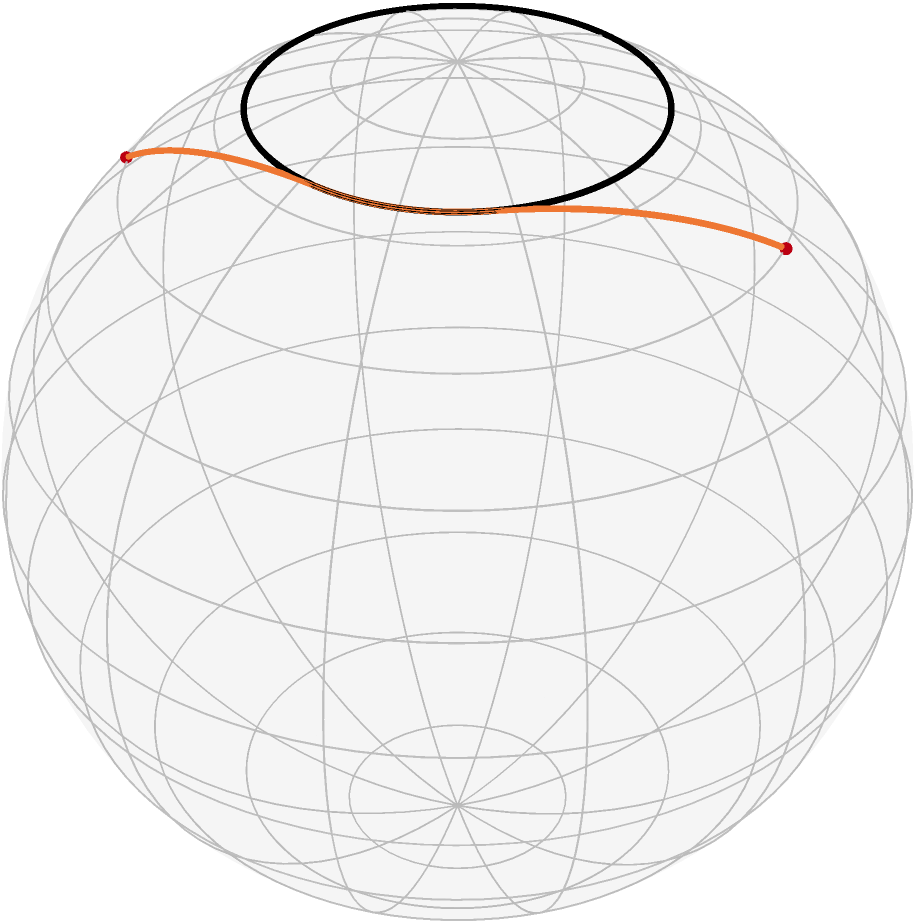}
    \end{minipage}
        \begin{minipage}[t]{0.49\textwidth}
        \centering
        \includegraphics[scale=0.5]{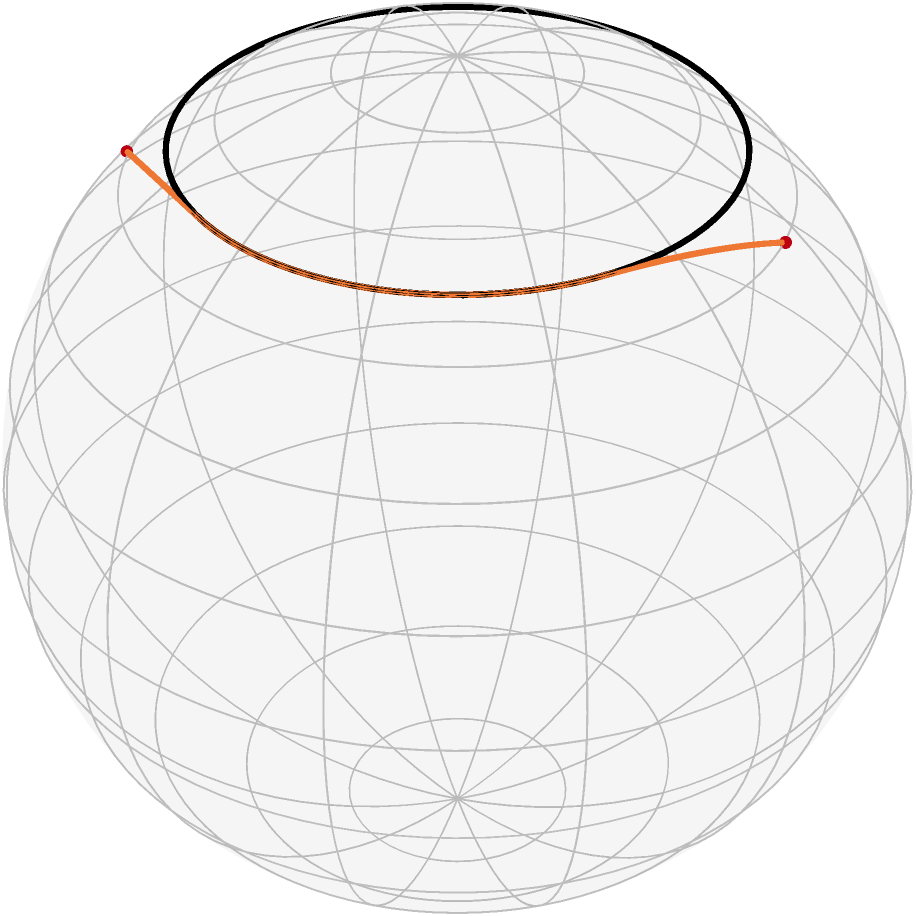}
    \end{minipage}
    \caption{Elastic geodesic avoiding the north pole cap with $h_{\mathrm{ref}} = 0.1$ (left) and $h_{\mathrm{ref}} = 0.2$ (right).}
    \label{fig:geodesicObstacle}
\end{figure}
%
%
\subsection{An Inextensible Elastic Rod}%
\label{sec:inextensibleRod}
As a simple example from continuum mechanics, we consider the computation of equilibrium states of an inextensible elastic rod.
First we provide the energetic formulation of the force free situation. For more details on the derivation of the model see~\cite{glowinski1989augmented}. Its solution via an SQP-method in Hilbert space was considered in~\cite{SchielaOrtiz:2021}.
We start with the following energy minimization problem
\begin{equation}\label{min_prob}
\begin{split}
\min_{y\in \M}&\frac{1}{2}\int_{0}^{1} \sigma(s) \langle \ddot y(s),\ddot y(s)\rangle ds,  \\
\M&=\{y\mid y\in H^2([0,1];\mathbb{R}^3),\dot y(s)\in \mathbb S^2  \, \mbox{on} \, [0,1] \}.
\end{split}
\end{equation}
Since $\dot y(s)\in \mathbb{S}^2$ for all $s\in [0,1]$ the rod is inextensible with fixed length $1$. The following boundary conditions are imposed:
\begin{align}\label{B.C}
\begin{split}
y(0)=y_a \in \mathbb{R}^3, \, \dot y(0)=v_a\in \mathbb{S}^2, \;
y(1)=y_b \in \mathbb{R}^3, \, \dot y(1)=v_b\in \mathbb{S}^2.
\end{split}
\end{align}
The quantity $\overline \sigma > \sigma(s)\ge \underline \sigma>0$ is the flexural stiffness of the rod, and $\dot y$, $\ddot y$ are the derivatives of $y$ with respect to $s\in[0,1]$.
Introducing $v(s)\coloneqq\dot y(s)$ we reformulate~\eqref{min_prob} as a mixed problem:
\begin{align}\label{eq:Mixed}
\begin{split}
\min_{(y,v)\in Y\times \mathcal V}& \frac{1}{2}\int_{0}^{1} \sigma \langle \dot v,\dot v\rangle \, ds
\quad \mbox{ s.t. } \quad \dot y-v =0\\
&Y=\{y\in H^2([0,1];\mathbb{R}^3)\,\, | \,\, y(0)=y_a,\,\, y(1)=y_b  \}, \\
&\mathcal V=\{v\in H^1([0,1];\mathbb{S}^2)\,\, |\,\, v(0)=v_a,\,\, v(1)=v_b  \}.
\end{split}
\end{align}
To derive equilibrium conditions for problem~\eqref{eq:Mixed} we define the Lagrangian function
 \begin{equation*}
  L(y,v,\lambda)
  =\int_{0}^{1} \frac{1}{2}  \sigma \left\langle \dot v,\dot v \right\rangle+\lambda (\dot y-v)\, ds
 \end{equation*}
using a Lagrangian multiplier $\lambda \in \Lambda \coloneqq L_2([0,1];\R^3)$. In the presence of an additional (possibly non-conservative) force-field $\omega \colon \R^3 \to (\R^3)^*$ we obtain the following equilibrium conditions via setting the derivatives of the Lagrangian, with force term added, to zero:
\begin{equation}\label{eq:systemRod}
\begin{split}
 \int_0^1 \omega(y)\phi_y +\lambda (\dot{\phi_y})\, ds&=0 \quad \text{ for all } \phi_y\in Y\\
\int_{0}^{1} \sigma \left\langle  \dot v,\dot{\phi_v}\right\rangle -  \lambda(\phi_v) ds&=0 \quad \text{ for all } \phi_v\in T_v \mathcal V\\
\int_{0}^{1}  \phi_\lambda(\dot y-v) ds&=0 \quad \text{ for all } \phi_\lambda\in \Lambda
\end{split}
\end{equation}
Hence, have to find a zero of the mapping
\begin{equation*}
 F \colon Y \times \mathcal V \times \Lambda \to Y^* \times T^*\mathcal V\times \Lambda^*,
\end{equation*}
defined by~\eqref{eq:systemRod}. For brevity we set $\X=Y \times \mathcal V \times \Lambda$ and $x=(y,v,\lambda)$ and obtain a mapping $F \colon \X \to T^*\X$.

To define Newton steps, we need retractions $R_x \colon T_x\X \to \X$ and vector transports $\VTforward{x} \in \Gamma(\L(\X \times T^*_x\X, T^*\X))$, which can be defined componentwise. Since $Y$ and $\Lambda$ are linear spaces, we can use the identity mappings for these components, while $\mathcal V$ is equipped with the retraction by normalization (cf.~\eqref{eq:pointwiseRetraction}) and the vector transport by orthogonal projections, in a similar way as in the geodesic problem.
Since we again have an embedded vector bundle, we can use the formula
\begin{equation*}
    Q_{F(x)}^*\circ F'(x)\delta x\,\phi = F_{\R^9}'(x)\delta x \,\phi + F(x)(\VTforward{x}'(x)\delta x\,\phi)
\end{equation*}
for $x \in \X$ and $\phi, \delta x \in T_x\X$. We denote the components of the tangent vectors by $\phi = (\phi_y, \phi_v, \phi_\lambda)$ and $\delta x = (\delta y, \delta v, \delta \lambda) \in Y \times T_v\mathcal{V} \times \Lambda$. The Euclidean derivative $F_{\R^9}'(x)$ of $F$ is given as
\begin{align*}
F_{\R^9}'(x)\delta x \, \phi&=
  \int_0^1 \omega'(y)\delta y \, \phi_y + \delta \lambda(\dot{\phi_y})\, ds\\
  &+\int_0^1 \sigma \langle \dot{\delta v}, \dot{\phi_v}\rangle  -\delta \lambda(\phi_v)\, ds\\
  &+
   \int_0^1\phi_\lambda(\dot{\delta y})  - \phi_\lambda(\delta v) \, ds.
\end{align*}
The part, introduced by the connection is given by
\begin{equation*}
    F(x)(\VTforward{x}'(x)\delta x \,\phi)=\int_0^1\langle \dot v,\left(P'(v)\delta v \, \phi_v\right)\dot{}\;\rangle-\lambda(P'(v)\delta v \, \phi_v)\, ds.
\end{equation*}
Observe that this term only contains contributions from the $\mathcal V$-component, since the other components are linear spaces.
\subsection*{Numerical results}
For our numerical example we set the flexural stiffness of the rod to $\sigma \equiv 1$ and assume that there is no additional force field, i.\,e.\ $\omega \equiv 0$. The boundary conditions are set to
\begin{align*}
\begin{split}
&y(0)=\begin{pmatrix}0 \\ 0 \\ 0 \end{pmatrix} \in \mathbb{R}^3, \, y(1)= \begin{pmatrix}0.8 \\ 0 \\ 0 \end{pmatrix}\in \mathbb{R}^3, \\
&v(0)=\frac{1}{\sqrt{5}}\begin{pmatrix}1 \\ 0 \\ 2 \end{pmatrix}\in \mathbb{S}^2, \, v(1)=\frac{1}{\sqrt{1.64}}\begin{pmatrix}1 \\ 0 \\ 0.8 \end{pmatrix}\in \mathbb{S}^2.
\end{split}
\end{align*}
On a given uniform grid on $[0,1]$ we discretize our function spaces $Y$ and $\mathcal V$ by piecewise linear continuous functions, and $\Lambda$ by discontinuous piecewise constant functions. The arising integrals are then computed numerically by the trapezoidal rule in a similar way as used in the geodesic problem. We applied Newton's method with affine-covariant damping to solve the equilibrium conditions~\eqref{eq:systemRod}. In Figure~\ref{fig:rod} the initial and the resulting rod can be seen. Figure~\ref{fig:ConvergenceAndStepsizesRod} shows the local superlinear convergence of Newton's method and the computed damping factors $\alpha$.
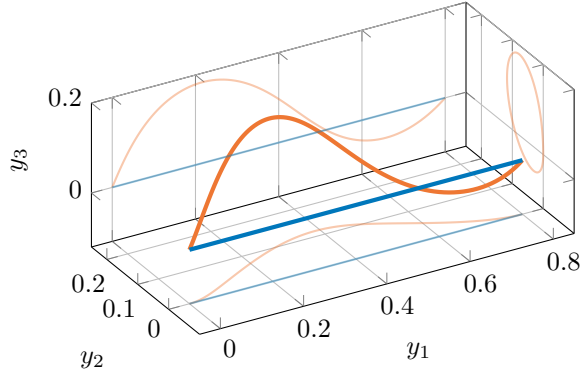
\begin{figure}[tbp]
    \begin{tikzpicture}
    \begin{axis}[
        view={-30}{35},
        width=8cm, height=6cm,
        grid=major,
        xlabel={$y_1$}, ylabel={$y_2$}, zlabel={$y_3$},
        ytick={0.0,0.1,...,0.3},
        zmin=-0.12, xmin=-0.05, ymin=-0.1,
        zmax=0.2, xmax=0.85, ymax=0.25,
        plot box ratio=2 1 1,
    ]

        \addplot3[ultra thick,color=TolVibrantOrange] table[x=x1, y=y1, z=z1,col sep=comma] {data/inextensible-rod-result.csv};

        \addplot3[thick,color=TolVibrantOrange,opacity=0.4] table[x, y expr=0.25, z=z1,col sep=comma] {data/inextensible-rod-result.csv};
        \addplot3[thick,color=TolVibrantOrange,opacity=0.4] table[x=x1, y=y1, z expr=-0.12,col sep=comma] {data/inextensible-rod-result.csv};

        \addplot3[thick,color=TolVibrantOrange,opacity=0.4] table[x expr=0.85, y=y1, z ,col sep=comma] {data/inextensible-rod-result.csv};

        \addplot3[ultra thick, color=TolVibrantBlue] table[ x=x1, y=y1, z=z1, col sep=comma] {data/inextensible-rod-data.csv};
        \addplot3[thick, color=TolVibrantBlue, opacity=0.4] table[x, y expr=0.25, z=z1, col sep=comma] {data/inextensible-rod-data.csv};

        \addplot3[thick, color=TolVibrantBlue, opacity=0.4] table[x=x1, y=y1, z expr=-0.12, col sep=comma] {data/inextensible-rod-data.csv};

    \end{axis}
    \end{tikzpicture}
    \caption{Inextensible rod with initial guess (blue) and computed solution (orange). The lines in light colors depict projections onto the coordinate planes (shadows).}
    \label{fig:rod}
\end{figure}

\begin{figure}[tbp]
    \begin{minipage}[t]{0.49\textwidth}
        \begin{tikzpicture}
        \begin{axis}[
            xlabel={Iteration},
            ylabel={$\Vert \delta x\Vert$},
            ymode=log,
            grid=both,
            width=9cm,
            height=8cm,
            scale=0.55,
            xtick={1,2,3,4,5,6,7,8,9},
            ytick distance=10^4
        ]
                \addplot table [
            col sep=comma,
            header=false,
            x expr=\coordindex+1,   
            y index=0             
            ]{data/norm_newton_direction_rod.csv};
        \end{axis}
        \end{tikzpicture}
    \end{minipage}
    \begin{minipage}[t]{0.49\textwidth}
                \begin{tikzpicture}
        \begin{axis}[
            xlabel={Iteration},
            ylabel={$\alpha$},
            grid=both,
            width=9cm,
            height=7.5cm,
            scale=0.6,
            xtick={1,2,3,4,5,6,7,8,9},
            ytick distance=0.1
        ]
                \addplot table [
            col sep=comma,
            header=false,
            x expr=\coordindex+1,   
            y index=0             
            ]{data/stepsize_rod.csv};
        \end{axis}
        \end{tikzpicture}
        \label{fig:stepsizesRod}
    \end{minipage}
    \caption{Superlinear convergence of Newton's method for $N=100$ discretization points (left) and damping factors $\alpha$ chosen by affine covariant damping strategy (right).}
    \label{fig:ConvergenceAndStepsizesRod}
\end{figure}
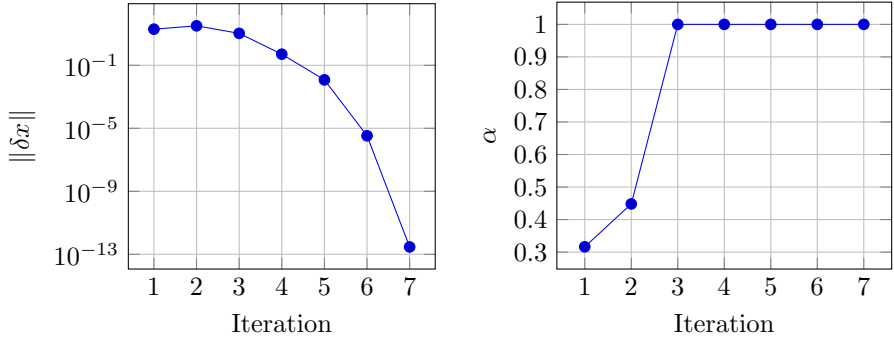
%
%
\section{Conclusion}
\label{sec:Conclusion}
We have elaborated, how Newton's method can be used to solve variational problems on manifolds. Our key observation was that these equations can be written as root finding problems for mappings into dual vector bundles. This allowed us to derive the differential geometric tools, necessary to implement a Newton method for this setting, together with an affine covariant damping scheme. We have further demonstrated that our algorithm is capable of solving a variety of problems in this class. Clearly, these problems are just simple examples for a large number of complex applications from various domains, to which our methods could be applied. This is, however, subject to future research.

\appendix

\printbibliography

\end{document}